\documentclass[12pt]{amsart}

\usepackage{fullpage}

\usepackage{amsmath}
\usepackage{amsfonts}
\usepackage{amssymb}
\usepackage{amsthm}
\usepackage{mathtools} 
\usepackage{latexsym}

\usepackage{enumerate}
\usepackage{cancel}
\usepackage{cases}
\usepackage{empheq}
\usepackage{multicol}

\usepackage{graphicx} 

\usepackage{wrapfig}

\usepackage{placeins}

\usepackage{mathrsfs}
\theoremstyle{plain}
\newtheorem{theorem}{Theorem}[section]
\newtheorem{proposition}[theorem]{Proposition}
\newtheorem{lemma}[theorem]{Lemma}

\theoremstyle{definition}
\newtheorem{definition}[theorem]{Definition}

\theoremstyle{remark}
\newtheorem{remark}[theorem]{Remark}

\numberwithin{equation}{section} 
\numberwithin{figure}{section}   

\usepackage[square,comma,numbers]{natbib}

\newcommand{\vect}[1]{\mathbf{#1}}

\newcommand{\bu}{\vect{u}}

\newcommand{\bx}{\vect{x}}

\newcommand{\field}[1]{\mathbb{#1}}
\newcommand{\nN}{\field{N}}
\newcommand{\nZ}{\field{Z}}

\newcommand{\nR}{\field{R}}

\newcommand{\nT}{\mathbb T}
\newcommand{\vphi}{\varphi}

\newcommand{\set}[1]{\left\{#1\right\}}

\newcommand{\pnt}[1]{\left(#1\right)}

\title[Reduced Kuramoto-Sivashinsky Equation]{On the well-posedness of an anisotropically-reduced two-dimensional Kuramoto-Sivashinsky equation}

\date{}

\author{Adam Larios}
\address[Adam Larios]{Department of Mathematics, 
                University of Nebraska--Lincoln,
        Lincoln, NE 68588-0130, USA}
\email[Adam Larios]{alarios@unl.edu}
\author{Kazuo Yamazaki}
\address[Kazuo Yamazaki]{Department of Mathematics and Statistics, 
                Texas Tech University, 
       Lubbock, TX, 79409, USA}
\email[Kazuo Yamazaki]{kyamazak@ttu.edu}


\begin{document}
\begin{abstract}
The Kuramoto-Sivashinsky equations (KSE) arise in many diverse scientific areas, and are of much mathematical interest due in part to their chaotic behavior, and their similarity to the Navier-Stokes equations.  However, very little is known about their global well-posedness in the 2D case.  Moreover, regularizations of the system (e.g., adding large diffusion, etc.) do not seem to help, due to the lack of any control over the $L^2$ norm.  In this work, we propose a new ``reduced'' 2D model that modifies only the linear part of (the vector form of) the 2D KSE in only one component.  This new model shares much in common with the 2D KSE: it is 4th-order in space, it has an identical nonlinearity which does not vanish in energy estimates, it has low-mode instability, and it lacks a maximum principle.  However, we prove that our reduced model is globally well-posed.  We also examine its dynamics computationally.  Moreover, while its solutions do not appear to be close approximations of solutions to the KSE, the solutions do seem to hold many qualitative similarities with those of the KSE.  We examine these properties via computational simulations comparing solutions of the new model to solutions of the 2D KSE.
\end{abstract}

\maketitle
\thispagestyle{empty}

\noindent
\section{Introduction}\label{secInt}
\noindent
The Kuramoto-Sivashinsky equation (KSE) appears frequently in diverse areas such as the study of instabilities in 
laminar flame fronts \cite{Sivashinsky_1977}, plasmas \cite{cohen1976non,laquey1975nonlinear}, reaction-diffusion systems \cite{Kuramoto_Tsuzuki_1975,Kuramoto_Tsuzuki_1976}, and the flow of fluid films on inclined planes \cite{sivashinsky1980vertical}. Indeed, under somewhat generic assumptions, it was shown in \cite{misbah1994secondary} that the dynamics of quite general physical systems obeying certain symmetries can be described in part by the KSE if a certain bifurcation point is exceeded, explaining the ubiquitous appearance of the equation.  Despite its prevalence, very little progress has been made in terms of its mathematical analysis for large times in dimensions higher than one, and major questions remain unanswered even in one dimension.  In this paper, we propose and analyze a hybrid version of the higher-dimensional KSE and Burgers equations that may shed light on the original system.  This new system has many characteristic features in common with the KSE: it is fourth-order in space, it has a low-mode instability, it has an advective-type nonlinearity, and the solution is not divergence-free.  However, unlike for the higher-dimensional KSE, we are able to provide a proof that this new system is globally well-posed, which is the main purpose of the present work.  We also provide computational simulations that compare the dynamics of the 2D KSE to the new 2D system.

The KSE was first derived in \cite{Kuramoto_Tsuzuki_1976,Sivashinsky_1977} (see also \cite{Kuramoto_1978,Sivashinsky_1980_stoichiometry,sivashinsky1980vertical}).  They are given in a domain $\Omega\subset\nR^n$ by
\begin{subequations}\label{KSE}
\begin{align}\label{KSE_eqn}
   \partial_t\bu+(\bu\cdot\nabla) \bu&= -\lambda\triangle\bu-\triangle^2 \bu && \text{ in }\Omega\times(0,T),
   \\\label{KSE_init}
   \bu(x,0) &=\bu^{in}(x) && \text{ in }\Omega,
\end{align}
\end{subequations}
with boundary conditions discussed below.  Here, $\lambda>0$ is a dimensionless constant.   One may also consider the scalar or ``integrated'' form given by
\begin{align}\label{KSE_scalar}
   \partial_t\vphi+\tfrac12|\nabla\vphi|^2&= -\lambda\triangle\vphi-\triangle^2 \vphi.
\end{align}
Note that by setting $\bu := \nabla\phi$, one formally recovers a solution to \eqref{KSE_eqn}.

In the one-dimensional case, with either periodic
($\Omega=\nT:=\nR/2\pi\nZ$) or full-space ($\Omega=\nR$) boundary conditions,
\eqref{KSE} is globally well posed, and in the periodic case has a finite-dimensional global attractor and an inertial manifold (see, e.g., \cite{Collet_Eckmann_Epstein_Stubbe_1993_Attractor,Collet_Eckmann_Epstein_Stubbe_1993_Analyticity,Constantin_Foias_Nicolaenko_Temam_1989,Constantin_Foias_Nicolaenko_Temam_1989_IM_Book,Foias_Nicolaenko_Sell_Temam_1985,Foias_Sell_Temam_1985,Foias_Sell_Titi_1989,Goluskin_Fantuzzi_2019,Goodman_1994,Grujic_2000_KSE,Hyman_Nicolaenko_1986,Ilyashenko_1992,Nicolaenko_Scheurer_Temam_1986,Otto_2009,Robinson_2001,Tadmor_1986,Temam_1997_IDDS,Wittenberg_2014_DCDSA} and the references therein).  In particular, the existence and uniqueness of the solution in the one-dimensional case is shown in \cite{Nicolaenko_Scheurer_1984}; we also refer to \cite{Nicolaenko_Scheurer_Temam_1985} for a result on the finite-dimensionality result using the notion of determining modes. Large-time behavior was also studied for the so-called ``Burgers-Sivashinsky'' equation, $\bu_t+\bu\cdot\nabla \bu = \bu + \triangle \bu$ in \cite{Goodman_1994,Molinet_2000_2D_BS}.  It was also shown in \cite{Cao_Titi_2006_KSE} that the only steady-state solutions to \eqref{KSE} in either $\nR^n$ or $\nT^n$, $n=1,2$, are constant functions.  The question of the global well-posedness of \eqref{KSE} for $n\geq 2$ in the periodic case, or $\nR^n$ is still open in general; however, in dimensions $n=2$ and $3$ for the case of radially symmetric initial data in an annular domain, global well-posedness was proved in \cite{Bellout_Benachour_Titi_2003}, assuming homogeneous Neumann boundary conditions.  On the other hand, in \cite{Pokhozhaev_2008} (see also \cite{Galaktionov_Mitidieri_Pokhozhaev_2008}) it was shown that, under a certain (seemingly non-physical) choice of third-order boundary conditions, for any dimension $n\geq 1$, solutions to \eqref{KSE} develop a singularity in finite time for a certain class of initial conditions.  These issues were discussed in \cite{Larios_Titi_2015_BC_Blowup}, where it was also shown that global well-posedness holds in the one-dimensional case, with a different choice of third-order boundary conditions.  The physical boundary conditions for \eqref{KSE} are given by
   $\bu \equiv \triangle \bu \equiv \mathbf{0} \text{ on }\partial\Omega.$
Currently,  the question of global existence of solutions to \eqref{KSE} under the physical boundary conditions, even in the 1D case, remains open.   Moreover, for $n\geq2$, the question of global well-posedness of \eqref{KSE} in the periodic case, or in the full space $\nR^n$, is also a challenging open question.  However, short-time existence (but not uniqueness) of solutions in Gevrey spaces in the case $\Omega=\nR^n$ for arbitrary dimension $n\in\nN$ was proven in \cite{Biswas_Swanson_2007_KSE_Rn}, and recently, in \cite{Ambrose_Mazzucato_2018}, it was shown that, so long as there are no linearly growing modes, then for sufficiently small initial data in a certain function space based on the Wiener algebra, global existence holds.  We also mention \cite{Sell_Taboada_1992,Benachour_Kukavica_Rusin_Ziane_2014_JDDE_2DKSE}, which studied global existence and attractors in 2D thin domains.

We remark that (\ref{KSE}) may be written in component form as 
\begin{subequations}
\begin{align}
& \partial_{t} u_{1} + (\bu\cdot\nabla) u_{1} = - \lambda \Delta u_{1} - \Delta^{2} u_{1}, \\
& \partial_{t} u_{2} + (\bu\cdot\nabla) u_{2} = - \lambda \Delta u_{2} - \Delta^{2} u_{2}. 
\end{align}
\end{subequations} 

In this paper, we propose and study the following two-dimensional system, which we call the \textit{reduced Kuramoto-Sivashinsky} equations (r-KSE)  written in terms of $\bu=(u_1,u_2)$.
\begin{subequations}\label{KSEr}
\begin{align}
\label{KSEr1}
   \partial_t u_1+(\bu\cdot\nabla) u_1&= \sigma u_1 + \nu\triangle u_1&& \text{ in }\Omega\times(0,T),
   \\
\label{KSEr2}
   \partial_t u_2+(\bu\cdot\nabla) u_2&= -\lambda\triangle u_2-\triangle^2 u_2&& \text{ in }\Omega\times(0,T),
   \\
\label{KSEr_init}
   \bu(\cdot,0) &=\bu^{in}(\cdot) = (u_1^{in},u_2^{in})(\cdot) && \text{ in }\Omega,
\end{align}
\end{subequations}
under periodic boundary conditions on the domain $\Omega = \nT^2 = \nR^2/2\pi\nZ^2 = [0,2\pi]^2$ and .  Here, $T>0$, $\nu>0$, $\sigma > 0$, and $\lambda>0$ are constants.  

\begin{remark}
Note that \eqref{KSEr} no longer appears to arise from a scalar form of the equations such as \eqref{KSE_scalar}, and hence for \eqref{KSE_scalar}, there is no obvious analogue of the modification that takes \eqref{KSE} to \eqref{KSEr}.  One possibility is to use a nonlinearity of the form $\tfrac12\nabla|\bu|^2$ instead of $\bu\cdot\nabla\bu$ in \eqref{KSEr}.  These are formally the same for the 2D KSE if one identifies $\bu\equiv\nabla\phi$, but for system \eqref{KSEr}, we make no assumption that $\bu\equiv\nabla\psi$ for any function $\psi$.  Rather than analyze both possible choices of nonlinearity, we made the arbitrary choice to focus on the nonlinearity $\bu\cdot\nabla\bu$ (this case is slightly more involved, since the $\int_\Omega\bu\,dx$ is no longer preserved by the flow), but results similar to those in this paper can also be proven for the nonlinearity $\tfrac12\nabla|\bu|^2$ 
using nearly identical arguments to those made below.

 We also note that it is clear that if ones switches the roles of $u_1$ and $u_2$ in \eqref{KSEr}, symmetric results to those in the present work hold.  There are several possibilities for 3D and higher-dimensional generalizations of the anisotropic reduction of \eqref{KSE} to \eqref{KSEr}.  The authors plan to investigate these questions in a future work.
\end{remark}


\begin{remark}
 A different modification of the 2D KSE was studied in \cite{Pinto_1998_thesis,Pinto_2001_Gevrey}; however, this was with a drastically simplified nonlinearity ($uu_x$ rather than $\bu\cdot\nabla\bu$) which vanishes in energy estimates.  It is our view that the central difficulty of the higher-dimensional KSE is that the nonlinearity does not vanish in $L^2$ energy estimates, analogous to the vorticity stretching for the 3D Navier-Stokes equations (NSE) term not vanishing in $L^2$ estimates of the vorticity.  We note that in system \eqref{KSEr} proposed above, the nonlinearity is identical to the nonlinearity in the 2D KSE, and hence does not vanish in $L^2$ energy estimates.  See also \cite{Boling_Fengqiu_1993_JPDE,Ioakim_Smyrlis_2016,Tomlin_Kalogirou_Papageorgiou_2018} for some other variations on KSE.
\end{remark}

We note that, as remarked upon above, one of the main obstacles in tackling the global well-posedness of the KSE (\ref{KSE}) in the 2D case is that even though the following one-dimensional integrals vanish, 
\begin{equation}\label{buuu_1D_zero}
\int_{\Omega} u_{i} \frac{\partial u_{i}}{\partial x_{i}}  u_{i}\, dx_i 
= \frac{1}{3}\int_{\Omega}  \frac{\partial }{\partial x_{i}}(u_i^3)\,dx_i 
=0,\qquad i=1,\ldots,n,
\end{equation} 
(which is the crucial fact that allows one to prove global well-posedness in 1D), such a result does not hold for the full nonlinearity in dimensions $n \geq 2$:
\begin{equation}\label{buuu_not_zero}
\int_{\Omega} (\bu\cdot\nabla) \bu \cdot \bu\, d\bx \neq 0,
\end{equation} 
since $\bu$ is not divergence free. This is reminiscent of the situation of the Burgers' equation in contrast to the Navier-Stokes equations; $\int_{\Omega} (\bu\cdot\nabla) \bu \cdot \bu \,d\bx = 0$ in the latter case due to the divergence-free condition while this integral is nonzero in general for the former case. With that in mind, we partially follow the work of \cite{Pooley_Robinson_2016} in the proof of our main result Theorem 3.2. 

Let us also point out that even if the initial data is mean-zero, such a property is not preserved through evolution of (\ref{KSEr}). Again, this is actually valid for the one-dimensional KSE (\ref{KSE}) but not for the two-dimensional KSE (\ref{KSE}) or (\ref{KSEr}), because 
\begin{equation*}
\int_{\Omega} (\bu\cdot\nabla) \bu\, d\bx \neq \mathbf{0}
\end{equation*} 
(which is also in contrast to the case of the NSE). This creates various difficulty such as the lack of applicability of Poincar$\acute{\mathrm{e}}$ inequality. Pooley and Robinson in \cite{Pooley_Robinson_2016} overcame such a difficulty using a bound on the moment of the solution $\bu$ \cite[Lemma 2]{Pooley_Robinson_2016}. We can obtain the analogous result with which we may use Poincar$\acute{\mathrm{e}}$ inequality. However, the computations become rather lengthy. In fact, as we will see, we may overcome this difficulty essentially by doing the estimates in an inhomogeneous space instead of homogeneous space, i.e., an $H^{1}(\mathbb{T}^{2})$-estimate instead of $\dot{H}^{1}(\mathbb{T}^{2})$-estimate. 

\subsection{Some remarks on the equation}

There are many studies that consider modifications of an equation for which the global existence and/or uniqueness of solutions is an open question.  Such works, including the present work, often then show that the modified equation is well-posed.  There are at least two major reasons for such studies.  The first is that sometimes the modified equation can be seen as a better-behaved approximation of the original equation, and thus it may be of use, e.g., in numerical simulations or studies of the dynamics.  For instance, the development of $\alpha$-models for the 3D NSE and related $\alpha$-models (see, e.g., 
\cite{Berselli_Spirito_2017,Cao_Lunasin_Titi_2006,Chen_Foias_Holm_Olson_Titi_Wynne_1998_PF,Rebholz_Layton_2012_book}
and the references therein).

The second reason is more subtle.  A modification of the equation can be seen as a way to try to understand something about the mechanisms underlying the dynamics predicted by the equations.  For instance, let us consider the abstract system
\begin{align}\label{abstract_eqn}
 \partial_t\bu+(\bu\cdot\nabla) \bu = \mathcal{N}(\bu),
\end{align}
where $\mathcal{N}$ is an operator that may be nonlinear and nonlocal.  For instance, if $\mathcal{N}(\bu)=-\lambda\Delta\bu - \Delta^2\bu$, then \eqref{abstract_eqn} is the KSE.  If  $\mathcal{N}(\bu)=\nu\Delta\bu - \nabla p$ (where the pressure $p=\Delta^{-1}\nabla\cdot((\bu\cdot\nabla) \bu)$ and $\Delta^{-1}$ is taken with respect to appropriate boundary conditions), then \eqref{abstract_eqn} is the NSE.  Let us consider the 2-dimensional case for the moment.  If $\mathcal{N}(\bu)=0$, then this equation is the inviscid Burgers equation, which is well-known to blow up in finite time.  If $\mathcal{N}(\bu)=-\nabla p$ (yielding the 2D Euler equations) or $\mathcal{N}(\bu)=\nu\Delta\bu$ (yielding the 2D viscous Burgers equations), then well-posedness is restored.  In the first case, this is due to the pressure ``weakening'' the nonlinearity by causing $\bu$ to be divergence-free, preventing \eqref{buuu_not_zero} and similarly weakening the nonlinearity in higher-order estimates, ultimately allowing proofs of global well-posedness to go through.  In the case of the 2D (or higher dimensional) viscous Burgers equation, the situation is quite different due to \eqref{buuu_not_zero}; however, as observed by O. Ladyzhenskaya (see, e.g., \cite{Ladyzhenskaya_1968} and the discussion in \cite{Pooley_Robinson_2016}), a maximum principle can be found for this system, which prevents the nonlinearity from forming arbitrarily large gradients.  This maximum principle is destroyed by adding a pressure gradient, as in the case of the Euler or the NSE), or by adding a higher-order diffusion, as in the case of the so-called ``hyper-viscous Burgers equation''  where $\mathcal{N}(\bu)=-\nu^2\Delta^2\bu$ (formally, this is \eqref{KSE_eqn} with $\lambda=0$) as pointed out in \cite{Larios_Titi_2015_BC_Blowup}.  In the 3D case, the pressure gradient weakens the nonlinearity in the sense that it prevents \eqref{buuu_not_zero}, but it no longer has a similar effect on higher-order estimates.  Moreover, it still destroys the maximum principle, so it is no longer clear to what extent the pressure affects the nonlinearity, or even whether its net effect is to weaken or strengthen the nonlinearity.  

From this perspective, one can see the reason for the interest in the multi-dimensional KSE (or even the multi-dimensional hyper-viscous Burgers equation).  It provides a setting in which the nonlinearity is rather strong (i.e., not ``weakened'' by the pressure or the maximum principle), and where the formation of arbitrarily large gradients is checked only by hyperdiffusion.  Interesting results in this direction appear even in the 1D case.  For instance, in \cite{Kostianko_Titi_Zelik_2018}, it is shown that adding a large dispersion term to the 1D KSE weakens the nonlinearity in the sense that the dispersion mechanism disperses large gradients as they begin to form, keeping energy in the lower modes where it increased by the low-mode instability in the KSE more than it is decreased by the hyperdiffusion.

Thus, the reduced system \eqref{KSEr} we propose in this work is of interest in the sense that its solution also has no maximum principle at all, while we point out that it has maximum principle in $L^{\infty}$-norm if $\sigma = 0$. Moreover, it does not sufficiently weaken the nonlinearity to prevent \eqref{buuu_not_zero} by, e.g., enforcing a divergence-free condition, but relies only on one-dimensional symmetries of the form \eqref{buuu_1D_zero}, present in all equations of the form \eqref{abstract_eqn}.  As we prove below, having only an exponential growth bound of a supremum norm of $u_1$ in Proposition \ref{exponential growth} is enough to tame nonlinearity sufficiently to obtain global well-posedness.


\section{Preliminaries}\label{secPre}
\noindent
We write $A \lesssim_{\alpha,\beta} B, A \approx_{\alpha, \beta} B$, etc.,  whenever there exists a constant $c = c(\alpha, \beta)$ such that $A\leq cB, A = cB$, respectively. For brevity we also write $\int \textbf{f} \triangleq \int_{\mathbb{T}^{2}} \textbf{f}(\textbf{x}) d\textbf{x}$ as well as $\partial_{j} \triangleq \frac{\partial}{\partial x_{j}}$ for $j \in \{1,2\}$. 
The standard $L^2$ inner-product is denoted by $(\cdot,\cdot)$.
We recall that, due to the periodic boundary condtions, for $\textbf{f}$ in suitable spaces, we may write 
\begin{equation*}
\textbf{f}(\textbf{x}) = \sum_{\textbf{k} \in \mathbb{Z}^{2}} \hat{\textbf{f}}(\textbf{k}) e^{i\textbf{k}\cdot \textbf{x}}, 
\hspace{3mm}
\lVert \textbf{f} \rVert_{L^{2}} \triangleq (\sum_{\textbf{k} \in \mathbb{Z}^{2}} \lvert \hat{\textbf{f}} (\textbf{k})  \rvert^{2} )^{\frac{1}{2}},
\end{equation*} 
and the inhomogeneous and homogeneous Sobolev norms 
\begin{equation*}
\lVert \textbf{f} \rVert_{H^{s}} \triangleq (\sum_{\textbf{k} \in \mathbb{Z}^{2}} (1+ \lvert \textbf{k} \rvert^{s} )^2 \lvert \hat{\textbf{f}} (\textbf{k}) \rvert^{2} )^{\frac{1}{2}}, \hspace{3mm} \lVert \textbf{f} \rVert_{\dot{H}^{s}} \triangleq (\sum_{\textbf{k} \in \mathbb{Z}^{2}} \lvert \textbf{k} \rvert^{2s} \lvert \hat{\textbf{f}}(\textbf{k}) \rvert^{2})^{\frac{1}{2}},
\end{equation*} 
respectively. Consequently $\lVert \textbf{f} \rVert_{H^{s}} \approx \lVert \textbf{f} \rVert_{\dot{H}^{s}} + \lVert \textbf{f} \rVert_{L^{2}}$. We denote by $\Lambda^{s} \triangleq (-\Delta)^{\frac{s}{2}}$ which is defined by its Fourier transform as $\widehat{\Lambda^{s} \textbf{f}}(\textbf{k}) = \lvert \textbf{k} \rvert^{s} \hat{\textbf{f}}(\textbf{k})$.   

We recall the Picard-Lindel\"of Theorem in Banach spaces, a proof of which can be found in, e.g., \cite[Theorem 3.1]{Majda_Bertozzi_2002}.
\begin{lemma}
\rm(Picard-Lindel\"of) 
Let $O \subset B$ be an open subset of a Banach space $B$ and let $F: O \mapsto B$ be a mapping that satisfies the following conditions 
\begin{enumerate}
\item $F(X)$ maps $O$ to $B$; 
\item $F$ is locally Lipschitz; i.e., for any $X \in O$ there exists $L > 0$ and an open neighborhood $U$ of $X$ in $O$ such that 
\begin{equation*}
\lVert F(\tilde{X}) - F(\overline{X}) \rVert_{B} \leq L \lVert \tilde{X} - \overline{X} \rVert_{B} 
\end{equation*} 
for all $\tilde{X}, \overline{X} \in U$. 
\end{enumerate} 
Then for any $X_{0} \in O$, there exists a time $T > 0$ such that 
\begin{equation*}
\partial_{t}X = F(X), \hspace{5mm} X \rvert_{t=0} = X_{0} \in O,
\end{equation*} 
has a unique solution $X \in C^{1}((-T, T); O)$. 
\end{lemma}

We also recall the Aubin-Lions-Simon Compactness Theorem, a proof of which can be found in, e.g., 
\cite[Theorem 5]{Simon_1987} (see also \cite[Lemma 4]{Simon_1990}).
\begin{lemma}
\rm(Aubin-Lions-Simon) 
Assume that $X, B, Y$ are all Banach spaces such that $X \subset B \subset Y$, where $X \hookrightarrow B$ compactly. Suppose $1 \leq p \leq \infty$, 
\begin{enumerate}
\item $F \triangleq \{f_{n}\}_{n}$ is bounded in $L^{p}([0, T]; X)$, 
\item $\frac{\partial F}{\partial t} \triangleq \{ \frac{\partial f_{n}}{\partial t} \}_{f_{n} \in F}$ is bounded in $L^{1}([0, T]; Y)$. 
\end{enumerate} 
Then $F$ is relatively compact in $L^{p} ([0, T]; B)$ and in $C([0, T]; B)$ if $p = \infty$. 
\end{lemma} 

 \section{Global Well-Posedness}\label{secGWP}
\noindent

We first write down the definition of a strong solution to the r-KSE (\ref{KSEr}). 
\begin{definition}
We call $\textbf{u} = (u_{1}, u_{2})$ a strong solution to (\ref{KSEr}) over a time interval $[0,T]$ if for any $\phi= (\phi_{1}, \phi_{2}) \in C^{\infty}(\mathbb{T}^{2})$, 
\begin{subequations}
\begin{align} 
(\partial_{t} u_{1}, \phi_{1}) + \nu (\nabla u_{1}, \nabla \phi_{1}) + (\textbf{u}\cdot\nabla u_{1}, \phi_{1}) = \sigma (u_{1}, \phi_{1}), \\
(\partial_{t} u_{2}, \phi_{2}) + (\Delta u_{2}, \Delta \phi_{2}) + (\textbf{u}\cdot\nabla u_{2}, \phi_{2}) = \lambda (\nabla u_{2}, \nabla \phi_{2}) 
\end{align} 
\end{subequations}  
for almost all $t \in [0,T]$, and 
\begin{subequations}
\begin{align}
& \bu \in L^{\infty} ([0, T]; H^{1}(\mathbb{T}^{2})), u_{1} \in L^{2}([0, T]; H^{2}(\mathbb{T}^{2})), u_{2} \in L^{2}([0, T]; H^{3}(\mathbb{T}^{2})),\\
& \bu \in C([0,T]; H^{s}(\mathbb{T}^{2})) \text{ for any } s \in [0,1), \partial_{t} \textbf{u} \in L^{2}([0, T]; H^{-1}(\mathbb{T}^{2})). 
\end{align}
\end{subequations} 
\end{definition} 

\begin{theorem}
Given any initial data $\textbf{u}^{in} \in H^{1}(\mathbb{T}^{2})$ such that $u_{1}^{in} \in L^{\infty}(\mathbb{T}^{2})$ and any $T > 0$, there exists a unique strong solution to (\ref{KSEr}) over $[0,T]$. 
\end{theorem} 

\begin{remark}
By symmetry, if the roles of $u_{1}$ and $u_{2}$ are reversed in (\ref{KSEr}), the analogous theorem clearly holds. 
\end{remark}

\section{Proof of Theorem 3.2}
\noindent
We consider a Galerkin approximation with $P_{n}$ being the projection onto the Fourier modes of order up to $n \in \mathbb{N} \cup \{0\}$: 
\begin{equation*}
P_{n}\textbf{u}(\textbf{x}) \triangleq  \sum_{\lvert \textbf{k} \rvert \leq n} \hat{\textbf{u}}(\textbf{k}) e^{i\textbf{x} \cdot \textbf{k}}. 
\end{equation*}
We let $\textbf{u}^{n} \triangleq P_{n}\bu \triangleq (P_{n} u_{1}, P_{n} u_{2}) $ and consider the following Galerkin-truncated system. 
\begin{subequations}\label{4.1}
\begin{align}
& \partial_{t} u_{1}^{n} + P_{n} ((\bu^{n} \cdot \nabla) u_{1}^{n}) = \sigma u_{1}^{n} + \nu \Delta u_{1}^{n}, \label{4.1a}\\
& \partial_{t} u_{2}^{n} + P_{n} ((\bu^{n} \cdot \nabla) u_{2}^{n}) = -\lambda \Delta u_{2}^{n} - \Delta^{2} u_{2}^{n}, \label{4.1b} \\
& \bu^{n}(\cdot, 0) \triangleq P_{n} \bu^{in}(\cdot) = P_{n} (u_{1}^{in}, u_{2}^{in}) = (u_{1}^{n}, u_{2}^{n})(0).  \label{4.1c}
\end{align}
\end{subequations} 
\begin{proposition}
Given initial data $\bu^{in} \in H^{1}(\mathbb{T}^{2})$, there exists $T = T(\lVert \textbf{u}^{in} \rVert_{H^{1}}) > 0$ such that the Galerkin approximation system (\ref{4.1}) has a solution $\textbf{u}^{n} \in L^{\infty} ([0,T]; H^{1}(\mathbb{T}^{2}))$ that satisfies $u_{1}^{n} \in L^{2}([0,T]; H^{2}(\mathbb{T}^{2}))$ and $u_{2}^{n} \in L^{2}([0,T]; H^{3}(\mathbb{T}^{2}))$; moreover, such bounds are independent of $n$. Additionally, $\partial_{t} u_{1}^{n} \in L^{2}([0,T]; L^{2}(\mathbb{T}^{2}))$ and $\partial_{t} u_{2}^{n} \in L^{2}([0,T]; H^{-1}(\mathbb{T}^{2}))$. Finally, if $T^{\ast}$ is the maximal existence time and $T^{\ast} < \infty$, then 
\begin{equation*}
\limsup_{t\to T^{\ast}} \lVert \bu^{n}(t) \rVert_{H^{1}} = + \infty.
\end{equation*} 
\end{proposition} 

\begin{proof} 
We rely on Lemma 2.1. In order to do so we define 
\begin{equation}\label{4.2}
F_{n} (\bu^{n}) \triangleq 
\begin{pmatrix}
F_{n,1} (\bu^{n})\\
F_{n,2}(\bu^{n})
\end{pmatrix}
\triangleq
\begin{pmatrix}
-P_{n} ((\bu^{n} \cdot \nabla) u_{1}^{n}) + \sigma u_{1}^{n} + \nu \Delta u_{1}^{n} \\
-P_{n} ((\bu^{n} \cdot \nabla) u_{2}^{n}) - \lambda \Delta u_{2}^{n} - \Delta^{2} u_{2}^{n} 
\end{pmatrix}.
\end{equation} 
In the estimates below, we make use of the following elementary facts: 
\begin{subequations}
\begin{align}
& \widehat{P_{n} \Lambda^{s} \textbf{f}}(\textbf{k}) = \widehat{ \Lambda^{s} P_{n}\textbf{f}}(\textbf{k}) \text{ so that } P_{n} \Lambda^{s} = \Lambda^{s} P_{n}, \label{4.3a}\\
&\lVert P_{n} \textbf{f} \rVert_{\dot{H}^{s}} \leq n^{s} \lVert \textbf{f} \rVert_{L^{2}}, \label{4.3b}\\
& \lVert P_{n} \textbf{f} \rVert_{\dot{H}^{s}} \leq  \lVert \textbf{f} \rVert_{\dot{H}^{s}}, \label{4.3c} \\
& \int (P_{n} \textbf{f})\cdot\textbf{g} = \int (P_{n}\textbf{f})\cdot(P_{n} \textbf{g}) = \int \textbf{f}\cdot (P_{n} \textbf{g}). \label{4.3d}
\end{align}
\end{subequations} 
Firstly, for $\textbf{f}^{n} \triangleq (f_{1}^{n}, f_{2}^{n}), \textbf{g}^{n} \triangleq (g_{1}^{n}, g_{2}^{n})$, we compute 
\begin{equation}\label{4.4}
\begin{split}
& \lVert F_{n,1} (\textbf{f}^{n}) - F_{n,1} (\textbf{g}^{n}) \rVert_{H^{1}}\\
=& \lVert -P_{n}( (\textbf{f}^{n} - \textbf{g}^{n}) \cdot \nabla f_{1}^{n}) - P_{n} (\textbf{g}^{n} \cdot \nabla (f_{1}^{n} - g_{1}^{n})) + \nu \Delta (f_{1}^{n} - g_{1}^{n}) + \sigma (f_{1}^{n} - g_{1}^{n}) \rVert_{H^{1}}\\
\lesssim& \lVert \textbf{f}^{n} - \textbf{g}^{n} \rVert_{L^{2}} \lVert \nabla f_{1}^{n} \rVert_{L^{\infty}} + \lVert \textbf{g}^{n} \rVert_{L^{\infty}} \lVert \nabla (f_{1}^{n} - g_{1}^{n}) \rVert_{L^{2}} + \nu n^{2} \lVert f_{1}^{n} - g_{1}^{n} \rVert_{L^{2}} + \sigma \lVert f_{1}^{n} - g_{1}^{n} \rVert_{L^{2}} \\
&+ n[\lVert \textbf{f}^{n} - \textbf{g}^{n} \rVert_{L^{2}} \lVert \nabla f_{1}^{n} \rVert_{L^{\infty}} + \lVert \textbf{g}^{n} \rVert_{L^{\infty}} \lVert \nabla (f_{1}^{n} - g_{1}^{n}) \rVert_{L^{2}} + \nu n^{2} \lVert f_{1}^{n} - g_{1}^{n} \rVert_{L^{2}} + \sigma \lVert f_{1}^{n} - g_{1}^{n} \rVert_{L^{2}}]\\
\lesssim& (1+n) [\lVert \textbf{f}^{n} - \textbf{g}^{n} \rVert_{L^{2}} \lVert f_{1}^{n} \rVert_{H^{3}} + \lVert \textbf{g}^{n} \rVert_{H^{2}} n \lVert f_{1}^{n} - g_{1}^{n} \rVert_{L^{2}} + \nu n^{2} \lVert f_{1}^{n} - g_{1}^{n} \rVert_{L^{2}} + \sigma \lVert f_{1}^{n} - g_{1}^{n} \rVert_{L^{2}}]\\
\lesssim&_{n, \lVert \textbf{f}^{n} \rVert_{L^{2}}, \lVert \textbf{g}^{n} \rVert_{L^{2}}} \lVert \textbf{f}^{n} - \textbf{g}^{n} \rVert_{L^{2}} 
\end{split}
\end{equation}
by (\ref{4.2}), H$\ddot{\mathrm{o}}$lder's inequality, the embedding of $H^{2}(\mathbb{T}^{2}) \hookrightarrow L^{\infty} (\mathbb{T}^{2})$, (\ref{4.3a}), (\ref{4.3b}) and (\ref{4.3c}).  
Secondly, we compute 
\begin{equation}\label{4.5} 
\begin{split}
& \lVert F_{n,2} (\textbf{f}^{n}) - F_{n,2}(\textbf{g}^{n}) \rVert_{H^{1}}\\
=& \lVert -P_{n}( ((\textbf{f}^{n} - \textbf{g}^{n}) \cdot \nabla ) f_{2}^{n}) - P_{n} ( (\textbf{g}^{n} \cdot \nabla) (f_{2}^{n} - g_{2}^{n}))\\
& \hspace{40mm} - \lambda \Delta (f_{2}^{n} - g_{2}^{n}) - \Delta^{2}(f_{2}^{n} - g_{2}^{n} ) \rVert_{H^{1}}\\
\lesssim& (1+n) [\lVert P_{n} (((\textbf{f}^{n} - \textbf{g}^{n}) \cdot \nabla) f_{2}^{n})\rVert_{L^{2}} + \lVert P_{n} ((\textbf{g}^{n} \cdot \nabla) (f_{2}^{n} - g_{2}^{n})) \rVert_{L^{2}}\\
& \hspace{50mm} + \lambda \lVert \Delta (f_{2}^{n} - g_{2}^{n} ) \rVert_{L^{2}} + \lVert \Delta^{2}(f_{2}^{n} - g_{2}^{n}) \rVert_{L^{2}}]\\
\lesssim& (1+n) [ \lVert \textbf{f}^{n} - \textbf{g}^{n} \rVert_{L^{2}} \lVert \nabla f_{2}^{n} \rVert_{L^{\infty}} + \lVert \textbf{g}^{n} \rVert_{L^{\infty}} \lVert \nabla (f_{2}^{n} - g_{2}^{n}) \rVert_{L^{2}}  + (\lambda n^{2} + n^{4}) \lVert \textbf{f}^{n} - \textbf{g}^{n} \rVert_{L^{2}}]\\
\lesssim& (1+n) [ \lVert \textbf{f}^{n} - \textbf{g}^{n} \rVert_{L^{2}} \lVert f_{2}^{n} \rVert_{H^{3}} + \lVert \textbf{g}^{n} \rVert_{H^{2}} n \lVert f_{2}^{n} - g_{2}^{n} \rVert_{L^{2}} + (\lambda n^{2} + n^{4}) \lVert \textbf{f}^{n} - \textbf{g}^{n} \rVert_{L^{2}}] \\
\lesssim&_{n, \lVert \textbf{f}^{n} \rVert_{L^{2}}, \lVert \textbf{g}^{n} \rVert_{L^{2}}} \lVert \textbf{f}^{n} - \textbf{g}^{n} \rVert_{L^{2}} 
\end{split}
\end{equation} 
by (\ref{4.2}), (\ref{4.3b}), (\ref{4.3c}), H$\ddot{\mathrm{o}}$lder's inequality and the embedding of $H^{2}(\mathbb{T}^{2}) \hookrightarrow L^{\infty}(\mathbb{T}^{2})$. 
Therefore, we conclude from (\ref{4.2}), (\ref{4.4}) and (\ref{4.5}) that 
\begin{equation}\label{4.6}
\lVert F_{n} (\textbf{f}^{n}) - F_{n}(\textbf{g}^{n}) \rVert_{H^{1}} \lesssim_{n, \lVert \textbf{f}^{n} \rVert_{L^{2}}, \lVert \textbf{g}^{n} \rVert_{L^{2}}} \lVert \textbf{f}^{n} - \textbf{g}^{n} \rVert_{H^{1}}. 
\end{equation} 
Thus, we see that $F_{n}$ is locally Lipschitz continuous in any open set $O^{M} \triangleq \{ \textbf{f} \in H^{1}(\mathbb{T}^{2}): \lVert \textbf{f} \rVert_{H^{1}} \leq M \}$. It is also clear that $F_{n}$ maps $O^{M}$ into $B = H^{1}$ by taking $\textbf{g}^{n} \equiv 0$ in (\ref{4.4}) and (\ref{4.5}). Thus, by Lemma 2.1, given $\textbf{u}^{in} \in H^{1}(\mathbb{T}^{2})$, there exists a unique solution 
\begin{equation}\label{4.7}
\textbf{u}^{n}\in C^{1} ([0, T_{n}), H^{1}(\mathbb{T}^{2}) \cap O^{M})^{2} 
\end{equation} 
for some $T_{n} > 0$. Now we take $L^{2}(\mathbb{T}^{2})$-inner products on (\ref{4.1a})-(\ref{4.1b}) with $(-\Delta u_{1}^{n}, -\Delta u_{2}^{n})$ to deduce 
\begin{equation}\label{4.8}
\begin{split}
& \frac{1}{2} \frac{d}{dt} (\lVert u_{1}^{n} \rVert_{\dot{H}^{1}}^{2} + \lVert u_{2}^{n} \rVert_{\dot{H}^{1}}^{2}) + \nu \lVert \Delta u_{1}^{n} \rVert_{L^{2}}^{2} + \lVert  u_{2}^{n} \rVert_{\dot{H}^{3}}^{2} \\
=& \int P_{n} ((\textbf{u}^{n} \cdot \nabla) u_{1}^{n}) \cdot \Delta u_{1}^{n} + \int P_{n} ((\bu^{n} \cdot \nabla) u_{2}^{n}) \cdot \Delta u_{2}^{n} - \sigma \int u_{1}^{n} \cdot \Delta u_{1}^{n}  + \lambda \lVert \Delta u_{2}^{n} \rVert_{L^{2}}^{2} \triangleq \sum_{i=1}^{4} I_{i}. 
\end{split}
\end{equation} 
As we pointed out in Remark 1.1, due to the lack of conserved quantity such as $L^{2}$-norm and the inaccessibility of Poincar$\acute{\mathrm{e}}$ inequality, this estimate alone will not work. Nevertheless, if we work on the $H^{1}(\mathbb{T}^{2})$-estimate instead of $\dot{H}^{1}(\mathbb{T}^{2})$-estimate, this difficulty may be overcome. For this purpose, we take $L^{2}$-inner products with $(u_{1}, u_{2})$ to obtain 
\begin{equation}\label{4.9} 
\begin{split}
& \frac{1}{2}  \frac{d}{dt} (\lVert u_{1}^{n} \rVert_{L^{2}}^{2} + \lVert u_{2}^{n} \rVert_{L^{2}}^{2}) + \nu \lVert \nabla u_{1}^{n} \rVert_{L^{2}}^{2} +\lVert \Delta u_{2}^{n} \rVert_{L^{2}}^{2} \\
=& - \int P_{n} ((\textbf{u}^{n} \cdot\nabla) u_{1}^{n}) \cdot u_{1}^{n} - \int P_{n} ((\bu^{n} \cdot \nabla) u_{2}^{n}) \cdot u_{2}^{n} + \sigma \lVert u_{1}^{n} \rVert_{L^{2}}^{2} + \lambda \lVert \nabla u_{2}^{n} \rVert_{L^{2}}^{2} \triangleq \sum_{i=1}^{4} II_{i}. 
\end{split}
\end{equation} 
We now start our estimates. Firstly, we compute 
\begin{equation}\label{4.10} 
\begin{split}
I_{1}\lesssim \lVert \bu^{n} \rVert_{L^{4}} \lVert \nabla u_{1}^{n} \rVert_{L^{4}} \lVert \Delta u_{1}^{n} \rVert_{L^{2}} 
\lesssim \lVert \bu^{n} \rVert_{H^{1}}^{\frac{3}{2}} \lVert \Delta u_{1}^{n} \rVert_{L^{2}}^{\frac{3}{2}}
\leq \frac{\nu}{2} \lVert \Delta u_{1}^{n} \rVert_{L^{2}}^{2} + c \lVert \bu^{n} \rVert_{H^{1}}^{6} 
\end{split}
\end{equation} 
where we used (\ref{4.8}), (\ref{4.3d}), (\ref{4.3a}), H$\ddot{\mathrm{o}}$lder's inequality, (\ref{4.3c}), the embedding of $H^{1}(\mathbb{T}^{2}) \hookrightarrow L^{4}(\mathbb{T}^{2})$, Gagliardo-Nirenberg inequality, and Young's inequality. Secondly, we compute 
\begin{equation}\label{4.11}
\begin{split}
I_{2} =& \int (\bu^{n} \cdot \nabla) u_{2}^{n} \cdot \Delta u_{2}^{n} \\
\leq& \lVert \bu^{n} \rVert_{L^{4}} \lVert \nabla u_{2}^{n} \rVert_{L^{4}} \lVert \Delta u_{2}^{n} \rVert_{L^{2}} \\
\lesssim& \lVert \bu^{n} \rVert_{H^{1}} \lVert u_{2}^{n} \rVert_{H^{1}}^{\frac{1}{2}} \lVert u_{2}^{n}  \rVert_{H^{2}}^{\frac{1}{2}} \lVert \Delta u_{2}^{n} \rVert_{L^{2}}
\end{split}
\end{equation} 
by (\ref{4.8}), H$\ddot{\mathrm{o}}$lder's inequality, (\ref{4.3c}), the embedding of $H^{1}(\mathbb{T}^{2}) \hookrightarrow L^{4}(\mathbb{T}^{2})$ and Gagliardo-Nirenberg inequality. 
Now it is clear that 
\begin{equation}\label{4.12}
\begin{split}
\lVert \Delta u_{2}^{n} \rVert_{L^{2}}^{2} = \sum_{\textbf{k} \in \mathbb{Z}^{2}} \lvert \textbf{k} \rvert^{4} \lvert \hat{u}_{2}^{n} \rvert^{2} \leq ( \sum_{\textbf{k} \in \mathbb{Z}^{2}} \lvert \textbf{k} \rvert^{2} \lvert \hat{u}_{2}^{n} \rvert^{2})^{\frac{1}{2}} (\sum_{\textbf{k} \in \mathbb{Z}^{2}} \lvert \textbf{k} \rvert^{6} \lvert \hat{u}_{2}^{n} \rvert^{2})^{\frac{1}{2}} = \lVert u_{2}^{n} \rVert_{\dot{H}^{1}} \lVert u_{2}^{n} \rVert_{\dot{H}^{3}}
\end{split}
\end{equation}
by H$\ddot{\mathrm{o}}$lder's inequality. Thus, we apply (\ref{4.12}) to (\ref{4.11}) and further bound by 
\begin{equation}\label{4.13}
\begin{split}
I_{2} \lesssim \lVert \bu^{n} \rVert_{H^{1}}^{\frac{3}{2}}(\lVert u_{2}^{n} \rVert_{L^{2}}^{\frac{1}{2}} + \lVert u_{2}^{n} \rVert_{\dot{H}^{2}}^{\frac{1}{2}}) \lVert u_{2}^{n} \rVert_{\dot{H}^{2}}\leq \frac{1}{4} \lVert u_{2}^{n} \rVert_{\dot{H}^{3}}^{2} + c(1+ \lVert \textbf{u}^{n} \rVert_{H^{1}}^{\frac{18}{5}})
\end{split}
\end{equation}
due to Young's inequality. Thirdly, we compute 
\begin{equation}\label{4.14}
I_{3} = \sigma \lVert \nabla u_{1}^{n} \rVert_{L^{2}}^{2}, \hspace{3mm} I_{4} \leq \lambda \lVert u_{2}^{n} \rVert_{\dot{H}^{1}} \lVert u_{2}^{n} \rVert_{\dot{H}^{3}} \leq \frac{1}{4} \lVert  u_{2}^{n} \rVert_{\dot{H}^{3}}^{2} +c \lVert \textbf{u}^{n} \rVert_{H^{1}}^{2} 
\end{equation}  
due to (\ref{4.8}), (\ref{4.12}) and Young's inequality. Fourthly, we compute 
\begin{equation}\label{4.15}
\begin{split}
II_{1} \leq \lVert \textbf{u}^{n} \rVert_{L^{4}} \lVert \nabla u_{1}^{n} \rVert_{L^{2}} \lVert u_{1}^{n} \rVert_{L^{4}} \lesssim \lVert \textbf{u}^{n} \rVert_{H^{1}}^{3} 
\end{split}
\end{equation} 
by (\ref{4.9}), H$\ddot{\mathrm{o}}$lder's inequality, (\ref{4.3c}) and the embedding of $H^{1}(\mathbb{T}^{2}) \hookrightarrow L^{4}(\mathbb{T}^{2})$. Fifthly, 
\begin{equation}\label{4.16}
\begin{split}
II_{2} \leq \lVert \bu^{n} \rVert_{L^{4}} \lVert \nabla u_{2}^{n} \rVert_{L^{2}} \lVert u_{2}^{n} \rVert_{L^{4}} \lesssim \lVert \bu^{n} \rVert_{H^{1}}^{3}
\end{split}
\end{equation} 
by (\ref{4.9}), H$\ddot{\mathrm{o}}$lder's inequality and (\ref{4.3c}). Finally, it is immediate that  
\begin{equation}\label{4.17}
II_{4} \lesssim \lVert \bu^{n} \rVert_{H^{1}}^{2}. 
\end{equation} 
Therefore, applying (\ref{4.10}), (\ref{4.13}), (\ref{4.14}), (\ref{4.15}), (\ref{4.16}) and (\ref{4.17}) to (\ref{4.8})-(\ref{4.9}) gives 
\begin{equation}\label{4.18} 
\frac{d}{dt} \lVert \bu^{n} \rVert_{H^{1}}^{2} + \nu \lVert \Delta u_{1}^{n} \rVert_{L^{2}}^{2} + \lVert u_{2}^{n} \rVert_{\dot{H}^{3}}^{2} \leq c(1+ \lVert \bu^{n} \rVert_{H^{1}}^{6} + \lVert \bu^{n} \rVert_{H^{1}}^{\frac{18}{5}} + \lVert \bu^{n} \rVert_{H^{1}}^{2} +\lVert \bu^{n} \rVert_{H^{1}}^{3}). 
\end{equation} 
This implies that there exists a constant $c \geq 0$ such that 
\begin{equation}\label{4.19}
\lVert \bu^{n}(t) \rVert_{H^{1}} 
\leq \frac{1+ \lVert \bu^{n} (0) \rVert_{H^{1}}}{[1 - 4ct (1+ \lVert \bu^{n}(0) \rVert_{H^{1}})^{4}]^{\frac{1}{4}}} - 1
\leq \frac{1+ \lVert \bu^{in} \rVert_{H^{1}}}{[1 - 4ct (1+ \lVert \bu^{in} \rVert_{H^{1}})^{4}]^{\frac{1}{4}}} - 1,
\end{equation} 
where we used the fact that $\lVert \bu^{n}(0) \rVert_{H^{1}} \leq \lVert \bu^{in} \rVert_{H^{1}}$, and the monotonicity of $g(x):=x/(1-\epsilon x^4)^{1/4}$ for small $x$ and small $\epsilon>0$. 
Thus, $H^{1}(\mathbb{T}^{2})$-norm does not blow up for all 
\begin{equation*}
t < T^{\ast} \triangleq \frac{1}{4c(1+ \lVert \bu^{in} \rVert_{H^{1}})^{4}}.
\end{equation*}
Hence, $T_{n} > T \triangleq \frac{T^{\ast}}{2}$ for all $n \in \mathbb{N}$ and 
\begin{equation}\label{4.20}
\bu^{n} \in L^{\infty}([0,T]; H^{1}(\mathbb{T}^{2})). 
\end{equation} 
We go back to (\ref{4.18}) and integrate in time to also deduce that 
\begin{equation}\label{4.21}
u_{1}^{n} \in L^{2}([0, T]; H^{2}(\mathbb{T}^{2})) \text{ and } u_{2}^{n} \in L^{2} ([0, T]; H^{3}(\mathbb{T}^{2}))
\end{equation} 
due to (\ref{4.20}). We also go back to (\ref{4.1a}) and directly take $L^{2}([0, T]; L^{2}(\mathbb{T}^{2}))$-norms to obtain 
\begin{equation}\label{4.22}
\begin{split}
\int_{0}^{T} \lVert \partial_{t} u_{1}^{n} \rVert_{L^{2}}^{2} d\tau 
\lesssim \sup_{t \in [0,T]} \lVert \nabla u_{1}^{n}(t) \rVert_{L^{2}}^{2} \int_{0}^{T} \lVert \bu^{n} \rVert_{H^{2}}^{2} d\tau + \int_{0}^{T} \lVert \Delta u_{1}^{n} \rVert_{L^{2}}^{2} + \lVert u_{1}^{n} \rVert_{L^{2}}^{2}  d\tau < \infty 
\end{split}
\end{equation} 
by the embedding of $H^{2}(\mathbb{T}^{2})\hookrightarrow L^{\infty} (\mathbb{T}^{2})$, (\ref{4.20}) and (\ref{4.21}). We also return to (\ref{4.1b}) and directly take $L^{2}([0,T]; H^{1}(\mathbb{T}^{2}))$-norms to obtain 
\begin{equation}\label{4.23} 
\begin{split}
&\int_{0}^{T} \lVert \partial_{t} u_{2}^{n} \rVert_{H^{-1}}^{2} d\tau \\
\lesssim& \sup_{t \in [0,T]} \lVert \nabla u_{2}^{n}(t) \rVert_{L^{2}}^{2}\int_{0}^{T} \lVert \bu^{n} \rVert_{H^{2}}^{2} d\tau + \int_{0}^{T} \lVert \nabla u_{2}^{n} \rVert_{L^{2}}^{2} + \lVert \Lambda^{3} u_{2}^{n} \rVert_{L^{2}}^{2} d\tau < \infty 
\end{split}
\end{equation} 
by the embedding of $H^{2}(\mathbb{T}^{2})\hookrightarrow L^{\infty}(\mathbb{T}^{2})$, (\ref{4.20}) and (\ref{4.21}). Finally, the fact that if $T^{\ast}$ is the maximal existence time and $T^{\ast} < \infty$, then $\limsup_{t\to T^{\ast}} \lVert \bu^{n}(t) \rVert_{H^{1}} = + \infty$ follows from how we deduced $T^{\ast} \triangleq \frac{1}{4c(1+ \lVert \bu^{in} \rVert_{H^{1}})^{4}}$ based on (\ref{4.19}). Indeed, if $\limsup_{t\to T^{\ast}} \lVert \bu^{n}(t) \rVert_{H^{1}} < \infty$, then we may obtain a solution on $[0,T]$, restart from $T$ until $T_{1}$ where $T_{1} < \frac{1}{4c(1+ \lVert \bu^{in}(T) \rVert_{H^{1}})^{4}}$; such a process may be repeated either for all time or until $\lVert \bu^{n} \rVert_{H^{1}}$ becomes infinite. This completes the proof of Proposition 4.1.   
\end{proof}

Using our results on the Galerkin approximation system, we will first deduce a local existence of a unique solution to (\ref{KSEr}). By Banach-Alaoglu theorem and weak compactness we obtain $\bu = (u_{1}, u_{2}) \in L^{\infty} ([0, T]; H^{1}(\mathbb{T}^{2}))$ such that $u_{1} \in L^{2}([0, T]; H^{2}(\mathbb{T}^{2})), u_{2} \in L^{2}([0, T]; H^{3}(\mathbb{T}^{2}))$ and a subsequence of $\{\bu^{n}\}_{n}$, which we  still denote by $\bu^{n}$, such that 
\begin{equation}\label{4.24}
\begin{split}
& \bu^{n} \to \bu \text{ weak}^{\ast} \text{ in } L^{\infty} ([0, T]; H^{1}(\mathbb{T}^{2})), \\
& u_{1}^{n} \to u_{1} \text{ weakly in } L^{2}([0, T]; H^{2}(\mathbb{T}^{2})), \\
& u_{2}^{n} \to u_{2} \text{ weakly in } L^{2}([0, T]; H^{3}(\mathbb{T}^{2}))
\end{split}
\end{equation} 
by (\ref{4.20}) and (\ref{4.21}). Now we let $p = 2, X = H^{2}(\mathbb{T}^{2}), Y = H^{-1}(\mathbb{T}^{2}), B = H^{s}(\mathbb{T}^{2})$ for $s \in [1,2)$ so that 
\begin{equation}\label{4.25}
\bu^{n} \to \bu \text{ strongly in } L^{2}(0, T; H^{s}(\mathbb{T}^{2})) \text{ for } s \in [1, 2)
\end{equation}
by Lemma 2.2, (\ref{4.21}), (\ref{4.22}) and (\ref{4.23}). Similarly letting $p = \infty, X = H^{1}(\mathbb{T}^{2}), Y = H^{-1}(\mathbb{T}^{2}), B = H^{s}(\mathbb{T}^{2})$ for $s \in [0, 1)$ shows that 
\begin{equation}\label{4.26}
\bu^{n} \to \bu \text{ strongly in } C([0,T]; H^{s}(\mathbb{T}^{2})) \text{ for } s \in [0, 1) 
\end{equation}
by Lemma 2.2, (\ref{4.20}), (\ref{4.22}) and (\ref{4.23}). Now we return to the Galerkin approximation (\ref{4.1a})-(\ref{4.1b}), take $L^{2}(\mathbb{T}^{2})$-inner products with $\{\textbf{w}_{j}\}_{j} = \{(w_{j,1}, w_{j,2}) \}_{j}  \subset H^{1}(\mathbb{T}^{2})$ that is dense in $H^{1}$ and multiply by $\psi: [0,T] \mapsto \mathbb{R}$ such that $\psi \in C^{1}([0,T])$ and $\psi(T) = 0$ to deduce 
\begin{subequations}
\begin{align}
& -\int_{0}^{T} (u_{1}^{n}, \psi'(t) w_{j,1}) dt + \nu \int_{0}^{T} (\nabla u_{1}^{n}, \nabla w_{j,1}) \psi(t) dt \nonumber\\
&+ \int_{0}^{T} (P_{n} ((\bu^{n} \cdot \nabla) u_{1}^{n}), w_{j,1} \psi(t) ) dt = (u_{1}^{n}(0), \psi(0) w_{j,1}) + \sigma \int_{0}^{T} (u_{1}^{n}, w_{j,1} \psi(t)) dt, \label{4.27a}\\
& - \int_{0}^{T} (u_{2}^{n}, \psi'(t) w_{j,2}) dt + \int_{0}^{T} (\Delta u_{2}^{n}, \Delta w_{j,2}) \psi(t) dt \nonumber\\
&+ \int_{0}^{T} (P_{n}((\bu^{n} \cdot \nabla) u_{2}^{n}), w_{j,2} \psi(t)) dt = (u_{2}^{n}(0), \psi(0) w_{j,2}) + \lambda \int_{0}^{T} (\nabla u_{2}^{n}, \nabla w_{j,2}) \psi(t) dt. \label{4.27b}
\end{align}
\end{subequations}
Firstly, 
\begin{equation}\label{4.28}
\begin{split}
& \int_{0}^{T} (u_{1}^{n}, \psi'(t) w_{j,1}) dt - \int_{0}^{T} (u_{1}, \psi'(t) w_{j,1}) dt \\
\leq& \sup_{t \in [0,T]} \lVert u_{1}^{n} - u_{1} \rVert_{L^{2}} \int_{0}^{T} \lVert w_{j,1} \rVert_{L^{2}} \lvert \psi'(t) \rvert dt \to 0 
\end{split}
\end{equation} 
by H$\ddot{\mathrm{o}}$lder's inequality and (\ref{4.26}). Identically we can show that  
\begin{equation}\label{4.29}
\int_{0}^{T} (u_{2}^{n}, \psi'(t) w_{j,2}) dt \to \int_{0}^{T} (u_{2}, \psi'(t) w_{j,2}) dt 
\end{equation} 
as $n\to\infty$. Next, 
\begin{equation}\label{4.30}
\begin{split}
& \nu \int_{0}^{T} (\nabla u_{1}^{n}, \nabla w_{j,1}) \psi(t) dt - \nu \int_{0}^{T} (\nabla u_{1}, \nabla w_{j,1}) \psi(t) dt \\
\leq& \nu( \int_{0}^{T} \lVert \nabla (u_{1}^{n} - u_{1}) \rVert_{L^{2}}^{2} dt)^{\frac{1}{2}}(\int_{0}^{T} \lVert \nabla w_{j,1} \rVert_{L^{2}}^{2} \psi^{2}(t) dt )^{\frac{1}{2}} \to 0 
\end{split}
\end{equation} 
as $n\to\infty$ by H$\ddot{\mathrm{o}}$lder's inequality and (\ref{4.25}). Identically we can show  
\begin{equation}\label{4.31}
\lambda \int_{0}^{T} (\nabla u_{2}^{n}, \nabla w_{j,2}) \psi(t) dt \to \lambda \int_{0}^{T} (\nabla u_{2}, \nabla w_{j,2}) \psi(t) dt 
\end{equation} 
as $n\to\infty$. On the other hand, we have 
\begin{equation}\label{4.32}
\begin{split}
& \int_{0}^{T} (\Delta u_{2}^{n}, \Delta w_{j,2}) \psi(t) dt - \int_{0}^{T} (\Delta u_{2}, \Delta w_{j,2}) \psi(t) dt \\
=& \int_{0}^{T} ((u_{2}^{n} - u_{2}), \Delta^{2} w_{j,2}) \psi(t) dt \to 0 
\end{split}
\end{equation} 
as $n\to\infty$ due to (\ref{4.24}) and that $(\Delta^{2} w_{j}) \psi(t) \in L^{2}(0, T; H^{-3}(\mathbb{T}^{2}))$. Next, 
\begin{equation}\label{4.33}
\begin{split}
& \int_{0}^{T} (P_{n} ((\bu^{n} \cdot \nabla) u_{1}^{n}), w_{j,1} \psi(t)) dt - \int_{0}^{T} ((u\cdot\nabla) u_{1}, w_{j,1} \psi(t)) dt \\
=& \int_{0}^{T} ((P_{n} - Id) (\bu^{n} \cdot \nabla) u_{1}^{n}, w_{j,1} \psi(t)) \\
&+ ((\bu^{n} - \bu) \cdot \nabla u_{1}^{n}, w_{j,1} \psi(t)) + ((\bu\cdot\nabla) (u_{1}^{n} - u_{1}), w_{j,1} \psi(t)) dt \triangleq \sum_{i=1}^{3} III_{i}. 
\end{split}
\end{equation} 
We estimate 
\begin{equation}\label{4.34}
\begin{split}
|III_{1}| \leq& \int_{0}^{T} \lVert \bu^{n} \rVert_{L^{4}} \lVert \nabla u_{1}^{n} \rVert_{L^{2}} \lVert (P_{n} - Id) w_{j,1} \rVert_{L^{4}} \lvert \psi(t) \rvert dt \\
\lesssim& (\sup_{t \in [0,T]} \lVert \bu^{n}(t) \rVert_{H^{1}}) (\sup_{t \in [0,T]} \lVert u_{1}^{n}(t) \rVert_{H^{1}}) \lVert (P_{n} - Id) w_{j,1} \rVert_{H^{1}} \int_{0}^{T} \lvert \psi(t) \rvert dt \to 0 
\end{split}
\end{equation} 
as $n\to\infty$ by (\ref{4.33}), (\ref{4.3d}), H$\ddot{\mathrm{o}}$lder's inequality, the embedding of $H^{1}(\mathbb{T}^{2}) \hookrightarrow L^{4}(\mathbb{T}^{2})$ and (\ref{4.20}). Secondly, 
\begin{equation}\label{4.35}
\begin{split}
|III_{2}| \leq& \int_{0}^{T} \lVert \bu^{n} - \bu \rVert_{L^{3}} \lVert \nabla u_{1}^{n} \rVert_{L^{2}} \lVert w_{j,1} \rVert_{L^{6}} |\psi(t)| dt \\
\leq& (\sup_{t \in [0,T]} \lVert (\bu^{n} - \bu)(t) \rVert_{H^{\frac{1}{2}}} ) \sup_{t \in [0,T]} (\lVert \nabla u_{1}^{n} (t) \rVert_{L^{2}}) \lVert w_{j,1} \rVert_{H^{1}} \int_{0}^{T} \lvert \psi(t) \rvert dt \to 0 
\end{split}
\end{equation}  
as $n\to\infty$ by (\ref{4.33}), H$\ddot{\mathrm{o}}$lder's inequality, the embeddings of $H^{\frac{1}{2}}(\mathbb{T}^{2}) \hookrightarrow L^{3}(\mathbb{T}^{2})$, $H^{1}(\mathbb{T}^{2})\hookrightarrow L^{6}(\mathbb{T}^{2})$, (\ref{4.20}) and (\ref{4.26}). Thirdly, 
\begin{equation}\label{4.36}
\begin{split}
|III_{3}| \leq& \int_{0}^{T} \lVert \bu \rVert_{L^{4}} \lVert \nabla (u_{1}^{n} - u_{1}) \rVert_{L^{2}} \lVert w_{j,1} \rVert_{L^{4}} \lvert \psi(t) \rvert dt \\
\lesssim& (\sup_{t \in [0,T]} \lVert \bu(t) \rVert_{H^{1}}) (\int_{0}^{T} \lVert \nabla (u_{1}^{n} - u_{1} ) \rVert_{L^{2}}^{2} dt)^{\frac{1}{2}} \lVert w_{j,1} \rVert_{H^{1}} (\int_{0}^{T} \lvert \psi(t) \rvert^{2} dt)^{\frac{1}{2}} \to 0 
\end{split}
\end{equation} 
as $n\to\infty$ by (\ref{4.33}) and H$\ddot{\mathrm{o}}$lder's inequality. Thus, applying (\ref{4.34})-(\ref{4.36}) to (\ref{4.33}) gives 
\begin{equation}\label{4.37}
\int_{0}^{T} (P_{n} (\bu^{n} \cdot \nabla) u_{1}^{n}, w_{j,1} \psi(t)) dt \to \int_{0}^{T} ((\bu\cdot\nabla) u_{1}, w_{j,1} \psi(t)) dt 
\end{equation} 
as $n\to\infty$. We did not rely on anything special about $u_{1}$ in the computations of (\ref{4.33})-(\ref{4.36}); thus, the same argument \textit{mutatis mutandis} shows that 
\begin{equation}\label{4.38}
\int_{0}^{T} (P_{n} (\bu^{n} \cdot \nabla) u_{2}^{n}, w_{j,2} \psi(t)) dt \to \int_{0}^{T} ((\bu\cdot\nabla) u_{2}, w_{j,2} \psi(t)) dt 
\end{equation} 
as $n\to\infty$. Finally, 
\begin{equation}\label{4.39}
\begin{split}
|(u_{1}^{n}(0), \psi(0) w_{j,1}) - (u_{1}(0), \psi(0) w_{j,1})|
\leq \lVert (P_{n} - Id) u_{1}^{in} \rVert_{L^{2}} \lVert w_{j,1} \rVert_{L^{2}} \psi(0) \to 0 
\end{split}
\end{equation} 
as $n\to\infty$ due to that (\ref{4.1c}). Identically we can show that 
\begin{equation}\label{4.40}
(u_{2}^{n}(0), \psi(0) w_{j,2}) \to (u_{2}^{in} (0), \psi(0) w_{j,2}) 
\end{equation} 
as $n\to\infty$. Thus, considering (\ref{4.28}), (\ref{4.29}), (\ref{4.30}), (\ref{4.31}), (\ref{4.32}), (\ref{4.37}), (\ref{4.38}), (\ref{4.39}), (\ref{4.40}), along with 
\begin{align*}
&\sigma \int_{0}^{T} (u_{1}^{n}, w_{j,1} \psi(t)) dt - \sigma \int_{0}^{T} (u_{1}, w_{j,1} \psi(t)) dt \\
\leq& \sigma \lVert u_{1}^{n} - u_{1} \rVert_{C([0,T]; L^{2}(\mathbb{T}^{2}))} \lVert w_{j,1}\rVert_{L^{2}} T \lVert  \psi \rVert_{C([0,T])} \to 0 
\end{align*} 
as $n\to\infty$ due to (\ref{4.26}), we may pass to the limit to obtain 
\begin{subequations}
\begin{align}
& -\int_{0}^{T} (u_{1}, \psi'(t) w_{j,1}) dt + \nu \int_{0}^{T} (\nabla u_{1}, \nabla w_{j,1}) \psi(t) dt \nonumber\\
&+ \int_{0}^{T} ((\bu\cdot\nabla) u_{1}, w_{j,1} \psi(t)) dt = (u_{1}^{in}, \psi(0) w_{j,1}) + \sigma \int_{0}^{T} (u_{1}, w_{j,1} \psi(t)) dt, \label{4.41a}\\
& -\int_{0}^{T} (u_{2}, \psi'(t) w_{j,2}) dt + \int_{0}^{T} (\Delta u_{2}, \Delta w_{j,2}) \psi(t) dt \nonumber\\
&+ \int_{0}^{T} ((\bu\cdot\nabla) u_{2}, w_{j,2} \psi(t)) dt = (u_{2}^{in}, \psi(0) w_{j,2}) + \lambda \int_{0}^{T} (\nabla u_{2}, \nabla w_{j,2}) \psi(t) dt \label{4.41b}
\end{align}
\end{subequations} 
for all $\textbf{w}_{j} = (w_{j,1}, w_{j,2}) \in H^{1}(\mathbb{T}^{2})$. It follows that (\ref{4.41a})-(\ref{4.41b}) continues to hold for any linear combinations of $\textbf{w}_{j} = (w_{j,1}, w_{j,2}) \in H^{1}(\mathbb{T}^{2})$ and thus for any $\textbf{v} = (v_{1}, v_{2}) \in H^{1}(\mathbb{T}^{2})$ by continuity and denseness of $\{ \textbf{w}_{j}\}_{j}$ in $H^{1}(\mathbb{T}^{2})$. Taking $\psi \in C_{c}^{\infty} ([0,T])$ also shows that 
\begin{subequations}
\begin{align}
&(\partial_{t} u_{1}, v_{1}) + \nu (\nabla u_{1}, \nabla v_{1}) + (\bu\cdot\nabla u_{1}, v_{1}) = \sigma (u_{1}, v_{1}), \label{4.42a}\\
& (\partial_{t} u_{2}, v_{2})+ (\Delta u_{2}, \Delta v_{2}) + (\bu\cdot\nabla u_{2}, v_{2}) = \lambda (\nabla u_{2}, \nabla v_{2}), \label{4.42b}
\end{align}
\end{subequations} 
holds in the distributional sense. We also multiply (\ref{4.42a})-(\ref{4.42b}) by $\psi \in C_{c}^{\infty}([0,T])$ such that $\psi(T) = 0$ and $\psi(0) = 1$ and integrate over $[0,T]$ to deduce 
\begin{subequations}
\begin{align}
& -\int_{0}^{T} (u_{1}, \psi'(t) v_{1}) dt + \nu \int_{0}^{T} (\nabla u_{1}, \nabla v_{1}) \psi(t) dt\nonumber\\
&+ \int_{0}^{T} ((\bu\cdot\nabla) u_{1}, v_{1} \psi(t)) dt = (u_{1}(0), \psi(0) v_{1}),\label{4.43a}\\
& - \int_{0}^{T} (u_{2}, \psi'(t) v_{2}) dt + \int_{0}^{T} (\Delta u_{2}, \Delta v_{2}) \psi(t) dt + \sigma \int_{0}^{T}(u_{1}, v_{1} \psi(t)) dt,  \nonumber\\
&+ \int_{0}^{T} ((\bu\cdot\nabla) u_{2}, v_{2} \psi(t)) dt = (u_{2}(0), \psi(0) v_{2}) + \lambda \int_{0}^{T} (\nabla u_{2}, \nabla v_{2}) \psi(t) dt.  \label{4.43b}
\end{align}
\end{subequations} 
In comparison of (\ref{4.41a})-(\ref{4.41b}) and (\ref{4.43a})-(\ref{4.43b}), we see that 
\begin{equation}\label{4.44}
\begin{split}
(u_{1}(0), \psi(0) v_{1}) = (u_{1}^{in}, \psi(0) v_{1}) \text{ and } (u_{2}(0), \psi(0) v_{2}) = (u_{2}^{in}, \psi(0) v_{2}).
\end{split} 
\end{equation} 
Thus, $(u_{1}(0) - u_{1}^{in}, v_{1}) =0$ and $(u_{2}(0) - u_{2}^{in}, v_{2}) = 0$ for any $v= (v_{1}, v_{2}) \in H^{1}(\mathbb{T}^{2})$, which implies that $u_{1}(0) = u_{1}^{in}$ and $u_{2}(0) = u_{2}^{in}$ in the sense of $L^2(\mathbb{T}^{2})$ functions.

Next, concerning uniqueness, suppose that $\bu = (u_{1}, u_{2})$ and $\textbf{v} = (v_{1}, v_{2})$ are both solutions to (\ref{KSEr}) with the same initial data. Letting $\textbf{w} = (w_{1}, w_{2}) \triangleq \bu-\textbf{v}$ gives 
\begin{subequations}
\begin{align}
& \partial_{t} w_{1} + (\textbf{u}\cdot\nabla) w_{1} + (\textbf{w}\cdot\nabla) v_{1} = \sigma w_{1} + \nu \Delta w_{1}, \label{4.45a}\\
& \partial_{t} w_{2} + (\textbf{u}\cdot\nabla) w_{2} + (\textbf{w}\cdot\nabla) v_{2} = - \lambda \Delta w_{2} - \Delta^{2} w_{2} \label{4.45b}
\end{align} 
\end{subequations}
so that taking $L^{2}(\mathbb{T}^{2})$-inner products with $(w_{1}, w_{2})$ leads to 
\begin{equation}\label{4.46}
\begin{split}
\frac{1}{2} \frac{d}{dt} \lVert \textbf{w} \rVert_{L^{2}}^2 +& \nu \lVert \nabla w_{1} \rVert_{L^{2}}^{2} + \lVert \Delta w_{2} \rVert_{L^{2}}^{2} 
= -( (\bu\cdot\nabla) w_{1}, w_{1})
+ ((\textbf{w}\cdot\nabla) v_{1}, w_{1}) \\
&+ ((\bu\cdot\nabla) w_{2}, w_{2})
+ ((\textbf{w}\cdot\nabla) v_{2}, w_{2}) 
+ \sigma \lVert w_{1} \rVert_{L^{2}}^{2} + \lambda \lVert \nabla w_{2} \rVert_{L^{2}}^{2} \triangleq \sum_{i=1}^{6} IV_{i}. 
\end{split}
\end{equation} 
We point out that in contrast to the case of the NSE, $IV_{1}$ and $IV_{3}$ in (\ref{4.46}) do not immediately vanish due to the lack of divergence-free property of $\bu$ in (\ref{KSE}). Now we estimate the terms in (\ref{4.46}) 
\begin{equation}\label{4.47}
\begin{split}
|IV_{1}| \leq \lVert \bu \rVert_{L^{\infty}} \lVert \nabla w_{1} \rVert_{L^{2}} \lVert w_{1} \rVert_{L^{2}} 
\leq \frac{\nu}{4} \lVert \nabla w_{1} \rVert_{L^{2}}^{2}+ c \lVert \bu \rVert_{H^{2}}^{2} \lVert w_{1} \rVert_{L^{2}}^{2}, 
\end{split}
\end{equation} 
\begin{equation}\label{4.48}
\begin{split}
|IV_{2}| \leq& \lVert \textbf{w}\rVert_{L^{2}} \lVert \nabla v_{1} \rVert_{L^{4}} \lVert w_{1} \rVert_{L^{4}} \\
\leq& c \lVert \textbf{w} \rVert_{L^{2}}^{2} \lVert \nabla v_{1} \rVert_{L^{2}} + \frac{\nu}{4} \lVert \nabla w_{1} \rVert_{L^{2}}^{2} + c \lVert \textbf{w} \rVert_{L^{2}}^{2} \lVert \Delta v_{1} \rVert_{L^{2}}^{2} 
\end{split}
\end{equation} 
and 
\begin{equation}\label{4.49}
\begin{split}
|IV_{3}| \leq \lVert \bu \rVert_{L^{\infty}} \lVert \nabla w_{2} \rVert_{L^{2}} \lVert w_{2} \rVert_{L^{2}}
\leq \frac{1}{4} \lVert \Delta w_{2} \rVert_{L^{2}}^{2} + c \lVert \bu \rVert_{H^{2}}^{\frac{4}{3}} \lVert w_{2} \rVert_{L^{2}}^{2}  
\end{split}
\end{equation} 
where we used H$\ddot{\mathrm{o}}$lder's inequality, the embeddings of $H^{2}(\mathbb{T}^{2})\hookrightarrow L^{\infty}(\mathbb{T}^{2})$ and $H^{1}(\mathbb{T}^{2})\hookrightarrow L^{4}(\mathbb{T}^{2})$, that 
\begin{equation}\label{4.50}
\lVert \nabla w_{2} \rVert_{L^{2}}^{2} = \sum_{\textbf{k} \in \mathbb{Z}^{2}} \lvert \textbf{k} \rvert^{2} \lvert \hat{w}_{2} (\textbf{k}) \rvert^{2} \leq (\sum_{\textbf{k} \in \mathbb{Z}^{2}} \lvert \hat{w}_{2}(\textbf{k}) \rvert^{2})^{\frac{1}{2}}(\sum_{\textbf{k} \in \mathbb{Z}^{2}} \lvert \textbf{k}\rvert^{4} \lvert \hat{w}_{2} (\textbf{k}) \rvert^{2} )^{\frac{1}{2}} \leq \lVert w_{2} \rVert_{L^{2}} \lVert w_{2} \rVert_{\dot{H}^{2}}
\end{equation} 
and Young's inequality. Next, 
\begin{equation}\label{4.51}
\begin{split}
|IV_{4}| \lesssim& \lVert \textbf{w} \rVert_{L^{2}} \lVert \nabla v_{2} \rVert_{H^{1}} \lVert w_{2} \rVert_{H^{1}} \\
\lesssim& \lVert \textbf{w} \rVert_{L^{2}}(\lVert \nabla v_{2} \rVert_{L^{2}} \lVert w_{2} \rVert_{L^{2}} + \lVert \Delta v_{2} \rVert_{L^{2}} \lVert \nabla w_{2} \rVert_{L^{2}})\\
\lesssim& \lVert \textbf{w} \rVert_{L^{2}}^{2} \lVert \nabla v_{2} \rVert_{L^{2}} + \lVert w \rVert_{L^{2}}^{\frac{3}{2}} \lVert \Delta v_{2} \rVert_{L^{2}} \lVert \Delta w_{2} \rVert_{L^{2}}^{\frac{1}{2}} \\
\leq& c \lVert \textbf{w} \rVert_{L^{2}}^{2} \lVert \Delta v_{2} \rVert_{L^{2}} + \frac{1}{8} \lVert \Delta w_{2} \rVert_{L^{2}}^{2} + c \lVert \textbf{w} \rVert_{L^{2}}^{2} \lVert \Delta v_{2} \rVert_{L^{2}}^{\frac{4}{3}} 
\end{split}
\end{equation} 
by (\ref{4.46}), H$\ddot{\mathrm{o}}$lder's inequality, the embedding of $H^{1}(\mathbb{T}^{2}) \hookrightarrow L^{4}(\mathbb{T}^{2})$, (\ref{4.50}) and Young's inequality. Finally, 
\begin{equation}\label{4.52}
\begin{split}
|IV_{6}| \leq \lambda \lVert w_{2} \rVert_{L^{2}} \lVert \Delta w_{2} \rVert_{L^{2}} 
\leq \frac{1}{8} \lVert \Delta w_{2} \rVert_{L^{2}}^{2} + c \lVert w_{2} \rVert_{L^{2}}^{2} 
\end{split}
\end{equation} 
by (\ref{4.46}), (\ref{4.50}) and Young's inequality. Therefore, we apply (\ref{4.47}), (\ref{4.48}), (\ref{4.49}), (\ref{4.51}) and (\ref{4.52}) to (\ref{4.46}) and conclude that 
\begin{equation}\label{4.53}
\begin{split}
& \frac{1}{2} \frac{d}{dt} \lVert \textbf{w} \rVert_{L^{2}}^{2} + \frac{\nu}{2} \lVert \nabla w_{1} \rVert_{L^{2}}^{2} + \frac{1}{2} \lVert \Delta w_{2} \rVert_{L^{2}}^{2} \\
\leq& c(1 + \lVert \textbf{u} \rVert_{H^{2}}^{2} + \lVert \textbf{v} \rVert_{H^{1}} + \lVert \textbf{v} \rVert_{H^{2}}^{2} + \lVert \bu \rVert_{H^{2}}^{\frac{4}{3}} + \lVert \textbf{v} \rVert_{H^{1}}^{\frac{4}{3}}) \lVert \textbf{w} \rVert_{L^{2}}^{2}. 
\end{split}
\end{equation} 
Gr\"onwall's inequality implies uniqueness, considering that $\textbf{u},\textbf{v} \in L^{2} ([0,T]; H^{2})$ due to (\ref{4.24}).  

Next, we finally extend our local solution globally in time. It suffices to prove a uniform bound on $H^{1}$-norm considering Proposition 4.1. We will need the following exponential growth bound on the supremum norm of $u_{1}$. 
\begin{proposition}\label{exponential growth}
Let $\bu = (u_{1}, u_{2})$ be a smooth solution to (\ref{KSEr}) over time interval $[0,T]$. Then for any $\alpha > \sigma > 0$, 
\begin{equation}\label{exp_growth_bound}
\sup_{t \in [0,T]} \lVert u_{1}(t) \rVert_{L^{\infty}} \leq \lVert u_{1}^{in} \rVert_{L^{\infty}}e^{2\alpha T}. 
\end{equation} 
\end{proposition} 

\begin{proof}
From (\ref{KSEr1}), we may fix $\alpha > \sigma > 0$, denote by $\phi(x,t) \triangleq e^{-\alpha t} u_{1}(x,t)$ and consider the equation of evolution of $\lvert \phi \rvert^{2}$.  A straight-forward computation yields,
\begin{equation}\label{4.55} 
\begin{split}
\partial_{t} \lvert \phi \rvert^{2} = - 2 \alpha \lvert \phi \rvert^{2} - (\bu\cdot\nabla) \lvert \phi \rvert^{2} + 2\nu \phi\Delta \phi  + 2 \sigma \lvert \phi \rvert^{2} 
\end{split}
\end{equation}  
due to (\ref{KSEr1}). Using the identity
\begin{equation}\label{4.56}
\begin{split}
\Delta \lvert \phi \rvert^{2} - 2 \lvert \nabla \phi \rvert^{2} = 2 \phi \Delta \phi 
\end{split}
\end{equation} 
we rewrite (\ref{4.55}) as 
\begin{equation}\label{4.57}
\partial_{t} \lvert \phi \rvert^{2} + 2 (\alpha-\sigma) \lvert \phi \rvert^{2} + (\bu\cdot\nabla) \lvert \phi \rvert^{2} -\nu\Delta \lvert \phi \rvert^{2} + 2\nu \lvert \nabla \phi \rvert^{2} = 0. 
\end{equation} 
Now suppose that $\lvert \phi \rvert^{2}$ has maximum at $(\textbf{x}^{\ast}, t^{\ast}) \in (0, T] \times \mathbb{T}^{2}$ and $\lvert \phi(\textbf{x}^{\ast}, t^{\ast}) \rvert^{2} \neq 0$. Then the left side of (\ref{4.57}) becomes strictly positive, leading to an immediate contradiction. Therefore, either $\lvert \phi(\textbf{x}^{\ast}, t^{\ast}) \rvert^{2} = 0$ and has a maximum at ($\textbf{x}^{\ast}, t^{\ast})$ or $\lvert \phi(\textbf{x}, t) \rvert^{2}$ has no maximum on $(0, T]\times \mathbb{T}^{2}$. If $\lvert \phi(\textbf{x}^{\ast}, t^{\ast}) \rvert^{2}  = 0$, then $\lvert \phi\rvert \equiv 0$ on $(0, T] \times \mathbb{T}^{2}$ so that $\phi = e^{-2 \alpha t} u_{1} \equiv 0$ indicating that $u_{1} \equiv 0$; hence (\ref{exp_growth_bound}) follows. On the other hand, if $\lvert \phi(\textbf{x}^{\ast}, t^{\ast} ) \rvert^{2} \neq 0$, because we know that the maximum exists on $[0,T]\times \mathbb{T}^{2}$, we must have 
\begin{equation*}
\lvert \phi(\textbf{x},t) \rvert^{2} \leq \lvert \phi(\textbf{x}^{\ast}, 0) \rvert^{2} 
\end{equation*} 
for some $\textbf{x}^{\ast} \in \mathbb{T}^{2}$ and all $t \in [0,T]$ and hence 
\begin{equation*}
\lvert u_{1}(\textbf{x},t) \rvert^{2} \leq e^{2\alpha t} \lvert u_{1}(\textbf{x}^{\ast}, 0) \rvert^{2} 
\end{equation*} 
for some $\textbf{x}^{\ast} \in \mathbb{T}^{2}$ and all $t \in [0,T]$. Therefore, 
\begin{equation*}
\lVert u_{1}(t) \rVert_{L^{\infty}}^{2} \leq e^{2\alpha t} \lVert u_{1}^{in} \rVert_{L^{\infty}}^{2} 
\end{equation*} 
for all $t \in [0,T]$, and thus (\ref{exp_growth_bound}) now follows. 
\end{proof}

\begin{remark}
If $\sigma =0$, then we could take $\alpha \searrow 0$ in (\ref{exp_growth_bound}) to deduce the maximum principle. Indeed, here we employed a proof that typically proves maximum principle and proved an exponential growth bound. This was necessary because even though a typical method to prove an exponential growth bound is an energy estimate and an application of Gr$\ddot{\mathrm{o}}$nwall's inequality (e.g., in the case of the 2D Euler equations), the structure of (\ref{KSEr}) does not allow the energy estimate to work due to the lack of conserved quantity to start with.  
\end{remark} 

The $L^{\infty} ([0,T]; L^{\infty}(\mathbb{T}^{2}))$-bound on $u_{1}$ leads to the following 
bound on $u_{2}$. 

\begin{proposition}\label{prop_u2_linf_bound}
Let $u = (u_{1}, u_{2})$ solve (\ref{KSEr}) over time interval $[0,T]$. Then 
\begin{equation}\label{4.58}
u_{2} \in L^{\infty} ([0, T]; L^{2}(\mathbb{T}^{2})) \cap L^{2}([0, T]; H^{2}(\mathbb{T}^{2})).  
\end{equation} 
\end{proposition} 

\begin{proof}
We take $L^{2}(\mathbb{T}^{2})$-inner products on (\ref{KSEr2}) with $u_{2}$ to obtain 
\begin{equation}\label{4.59}
\begin{split}
\frac{1}{2} \frac{d}{dt} \lVert u_{2} \rVert_{L^{2}}^{2} + \lVert \Delta u_{2} \rVert_{L^{2}}^{2} = - \int u_{1} \partial_{1} u_{2} u_{2} - \int u_{2} \partial_{2} u_{2} u_{2} - \lambda \int \Delta u_{2} u_{2}. 
\end{split}
\end{equation} 
We make use of the fact that $\int u_{2} (\partial_{2} u_{2}) u_{2} = \int \frac{1}{6} \partial_{2} (u_{2})^{3} = 0 $ and   estimate 
\begin{equation}\label{4.60}
\begin{split}
\frac{1}{2} \frac{d}{dt} \lVert u_{2} \rVert_{L^{2}}^{2} + \lVert \Delta u_{2} \rVert_{L^{2}}^{2} 
\leq& \lVert u_{1} \rVert_{L^{\infty}} \lVert \nabla u_{2} \rVert_{L^{2}} \lVert u_{2} \rVert_{L^{2}} + \frac{1}{4} \lVert \Delta u_{2} \rVert_{L^{2}}^{2} + c \lVert u_{2} \rVert_{L^{2}}^{2} \\
\leq& \frac{1}{2} \lVert \Delta u_{2} \rVert_{L^{2}}^{2} + c \lVert u_{2} \rVert_{L^{2}}^{2}
\end{split}
\end{equation} 
by H$\ddot{\mathrm{o}}$lder's inequality, Young's inequality, (\ref{exp_growth_bound}) and (\ref{4.50}). Subtracting $\frac{1}{2} \lVert \Delta u_{2} \rVert_{L^{2}}^{2}$ from both sides of (\ref{4.60}) and applying Gr\"onwall's inequality completes the proof of Proposition \ref{prop_u2_linf_bound}. 
\end{proof} 

We are almost ready to complete the $H^{1}(\mathbb{T}^{2})$-bound; however, we will see that we need to improve the $L^{\infty}([0,T]; L^{\infty}(\mathbb{T}^{2}))$-bound of $u_{1}$ to $L^{2}([0,T]; H^{1}(\mathbb{T}^{2}))$-bound as usual (see (\ref{4.64}). 

\begin{proposition}\label{prop_u1_L2H1_bound}
Let $\textbf{u} = (u_{1}, u_{2})$ solve (\ref{KSEr}) over time interval $[0,T]$. Then 
\begin{equation}\label{4.61}
u_{1} \in L^{2}([0, T]; H^{1}(\mathbb{T}^{2})).  
\end{equation} 
\end{proposition} 

\begin{proof}
We take $L^{2}(\mathbb{T}^{2})$-inner products on (\ref{KSEr1}) with $u_{1}$ to first rewrite 
\begin{align*}
\frac{1}{2} \frac{d}{dt} \lVert u_{1} \rVert_{L^{2}}^{2} + \nu \lVert \nabla u_{1} \rVert_{L^{2}}^{2} =& -\int (\bu\cdot\nabla) u_{1} u_{1} + \sigma \lVert u_{1} \rVert_{L^{2}}^{2} \\
=& -\int u_{1} \partial_{1} u_{1} u_{1} + u_{2} \partial_{2} u_{1} u_{1} + \sigma \lVert u_{1} \rVert_{L^{2}}^{2} \\
=& -\int u_{2} \frac{1}{2} \partial_{2} (u_{1})^{2}+ \sigma \lVert u_{1} \rVert_{L^{2}}^{2}  = \frac{1}{2} \int (\partial_{2} u_{2}) u_{1} u_{1} + \sigma \lVert u_{1} \rVert_{L^{2}}^{2}  
\end{align*} 
where we used \eqref{buuu_1D_zero}; this is crucial because we do not have any bound on the derivative of $u_{1}$ yet. Now we continue to bound by 
\begin{equation}\label{4.62}
\begin{split}
\frac{1}{2} \frac{d}{dt} \lVert u_{1} \rVert_{L^{2}}^{2} + \nu \lVert \nabla u_{1} \rVert_{L^{2}}^{2} \leq& \frac{1}{2} \lVert \nabla u_{2} \rVert_{L^{2}} \lVert u_{1} \rVert_{L^{\infty}} \lVert u_{1} \rVert_{L^{2}}  + \sigma \lVert u_{1} \rVert_{L^{2}}^{2} \\
\lesssim& \lVert \nabla u_{2} \rVert_{L^{2}} \lVert u_{1} \rVert_{L^{\infty}}^{2} + \lVert u_{1} \rVert_{L^{\infty}}^{2} 
\lesssim \lVert \nabla u_{2} \rVert_{L^{2}} + 1 
\end{split}
\end{equation} 
by H$\ddot{\mathrm{o}}$lder's inequality, the embedding of $L^{\infty}(\mathbb{T}^{2})\hookrightarrow L^{2}(\mathbb{T}^{2})$ and (\ref{exp_growth_bound}). Because $u_{2} \in L^{2}([0, T]; H^{2}(\mathbb{T}^{2}))$ by (\ref{4.58}),  integrating (\ref{4.62}) in time completes the proof of Proposition \ref{prop_u1_L2H1_bound}. 
\end{proof} 

Finally, the following proposition will complete the proof of Theorem 3.2. 

\begin{proposition}\label{prop_u1u2_bounds}
Let $\textbf{u} = (u_{1}, u_{2})$ solve (\ref{KSEr}) over time interval $[0,T]$. Then 
\begin{equation}\label{4.63}
\textbf{u}\in L^{\infty} ([0,T]; H^{1}(\mathbb{T}^{2})), u_{1} \in L^{2}([0,T]; H^{2}(\mathbb{T}^{2})), u_{2} \in L^{2} ([0,T]; H^{3}(\mathbb{T}^{2})). 
\end{equation} 
\end{proposition} 

\begin{proof}
We take $L^{2}(\mathbb{T}^{2})$-inner products on (\ref{KSEr}) with $(-\Delta u_{1}, -\Delta u_{2})$ to study 
\begin{equation}\label{4.64}
\begin{split}
&\frac{1}{2} \frac{d}{dt} \lVert \nabla u \rVert_{L^{2}}^{2} + \nu \lVert \Delta u_{1} \rVert_{L^{2}}^{2} + \lVert  u_{2} \rVert_{\dot{H}^{3}}^{2} \\
=& \int (\bu\cdot\nabla) u_{1} \Delta u_{1} + \int (\bu\cdot\nabla) u_{2} \Delta u_{2} + \sigma \lVert \nabla u_{1} \rVert_{L^{2}}^{2} + \lambda \lVert \Delta u_{2} \rVert_{L^{2}}^{2} \\
=& -\sum_{i,k=1}^{2} \int \partial_{k} u_{i} \partial_{i} u_{1} \partial_{k} u_{1} + \partial_{k} u_{i} \partial_{i} u_{2} \partial_{k} u_{2} \\
&+ \frac{1}{2} \sum_{i,k=1}^{2} \int \partial_{i} u_{i} \partial_{k} u_{1} \partial_{k} u_{1} + \partial_{i} u_{i} \partial_{k} u_{2} \partial_{k} u_{2} + \sigma \lVert \nabla u_{1} \rVert_{L^{2}}^{2} +  \lambda \lVert \Delta u_{2} \rVert_{L^{2}}^{2} \\
\lesssim& \lVert \nabla u \rVert_{L^{2}} \lVert \nabla u_{1} \rVert_{L^{4}}^{2} + \lVert \nabla u \rVert_{L^{2}} \lVert \nabla u_{2} \rVert_{L^{4}}^{2} + \lVert \nabla u_{1} \rVert_{L^{2}}^{2}+  \lVert \Delta u_{2} \rVert_{L^{2}}^{2} \\
\lesssim& \lVert \nabla u \rVert_{L^{2}} \lVert u_{1} \rVert_{H^{1}} \lVert u_{1} \rVert_{H^{2}}+ \lVert \nabla u \rVert_{L^{2}} \lVert  u_{2}   \rVert_{H^{1}}^{\frac{3}{2}} \lVert u_{2}  \rVert_{H^{3}}^{\frac{1}{2}} + \lVert \nabla u_{1} \rVert_{L^{2}}^{2} + \lVert \nabla u_{2} \rVert_{L^{2}} \lVert  u_{2} \rVert_{\dot{H}^{3}} \\
\leq& \frac{\nu}{2} \lVert \Delta u_{1} \rVert_{L^{2}}^{2} + \frac{1}{2} \lVert u_{2} \rVert_{\dot{H}^{3}}^{2} + c(1+ \lVert \nabla u \rVert_{L^{2}}^{2}) (1+ \lVert \nabla u \rVert_{L^{2}}^{2}). 
\end{split}
\end{equation} 
where we used H$\ddot{\mathrm{o}}$lder's inequality, (\ref{4.12}) and Gagliardo-Nirenberg inequality. Subtracting $\frac{\nu}{2} \lVert \Delta u_{1} \rVert_{L^{2}}^{2} + \frac{1}{2} \lVert  u_{2} \rVert_{\dot{H}^{3}}^{2}$ from both sides and relying on Gr\"onwall's inequality completes the proof of Proposition \ref{prop_u1u2_bounds} as $u_{1} \in L^{2}([0, T]; H^{1}(\mathbb{T}^{2}))$ and $u_{2} \in L^{2}([0, T]; H^{2}(\mathbb{T}^{2}))$ by (\ref{4.58}) and (\ref{4.61}). 
\end{proof} 

\section{Computational Results}
\noindent
In this section, we demonstrate the dynamical differences between the KSE system \eqref{KSE} and r-KSE system \eqref{KSEr} by looking at numerical simulations of the equations side-by-side.  We do not focus on the particular dynamics of the KSE system \eqref{KSE}, since this has been studied elsewhere.  See, e.g., \cite{Kalogirou_Keaveny_Papageorgiou_2015} for an in-depth computational study of the 2D KSE and \cite{Bezia_Mabrouk_2019} for a finite-difference scheme for the 2D KSE.  We also mention the computational study \cite{Tomlin_Kalogirou_Papageorgiou_2018}, which examines a 2D dispersive anisotropic version of the KSE, with nonlinearity $uu_x$.  

Note that \textit{we do not make any claims that solutions of the r-KSE are good approximations to solutions of the KSE}, but we are interested in it for phenomenological reasons, as discussed in the introduction.  Therefore, we present solutions to the KSE in comparison to solutions of the r-KSE to get some idea of their similarities and differences.

\subsection{Choice of Parameters}
The r-KSE system has two additional parameters that do not appear in the KSE model; namely $\nu>0$ and $\sigma\geq0$.  While this lends considerable freedom, it also greatly increases the parameter space that can be explored.  Hence, we wish to restrict the parameter space to some phenomenologically interesting region.  While we are not trying to approximate solutions of the KSE, we seek to capture some of its qualitative properties.  Thus, it makes sense to make the ``instability cut-off'' for the r-KSE match that of the KSE.  That is, if we compute the Fourier symbols of the linear operators in the r-KSE and KSE systems, namely $\sigma + \nu\Delta$ and $\lambda\Delta +\Delta^2$, we obtain $\sigma -\nu|\mathbf k|^2$ and $\lambda|\mathbf k|^2 -|\mathbf k|^4$, respectively.  Thus, the unstable modes occur exactly at those values of $\mathbf k$ where  $|\mathbf k|^2< \sigma/\nu $ (for $u_1$ in r-KSE) and $|\mathbf k|^2<\lambda$ (for KSE and $u_2$ in r-KSE).  Thus, to obtain the same instability cut-off for both equations, in our simulations, we set:
\begin{align}\label{param_cut_off}
 \nu = \sigma/\lambda\quad\text{ when }\sigma>0.
\end{align}
(Notice that there is no explicit contribution to the unstable modes from the $\sigma + \nu\Delta$ operator when $\sigma = 0$.)  Thus, equating the instability cut-offs eliminates one free parameter.

To further limit the parameter space, one can consider the ``total amount of instability'' contributed by the two linear operators, and equate these.  Toward this end, we compute:
\begin{align}
 \sum_{\{\mathbf k\in\nZ^2: |\mathbf k|^2<\sigma/\nu\}}(\sigma -\nu|\mathbf k|^2)
 \approx 
 \int_{B(0,\sqrt{\frac{\sigma}{\nu}})}(\sigma -\nu|\bx|^2)\,d\bx
 =
 \pi\frac{\sigma}{\nu} - 2\pi\nu\int_0^{\sqrt{\frac{\sigma}{\nu}}}r^3\,dr
 =
 \frac{\pi\sigma^2}{2\nu},
\end{align}
and similarly,
\begin{align*}
 \sum_{\{\mathbf k\in\nZ^2: |\mathbf k|^2<\lambda\}}(\lambda|\mathbf k|^2 -|\mathbf k|^4)
  &\approx 
 \int_{B(0,\sqrt{\frac{\sigma}{\nu}})}(\lambda|\bx|^2 -|\bx|^4)\,d\bx
=
 \pi\pnt{\frac{\sigma}{\nu}}^2\pnt{\frac{\lambda}{2} - \frac{1}{3}\frac{\sigma}{\nu}}
\end{align*}
Setting these equal, we find
\begin{align}
3\nu&=3\lambda - 2\frac{\sigma}{\nu}.
\end{align}
Combining this with \eqref{param_cut_off}, we find
\begin{align}\label{nu_sigma_values}
 \nu = \frac{\lambda}{3}\text{ and } \sigma = \frac{\lambda^2}{3}.
\end{align}
Thus, if we accept the equating done above (as well as the approximation of the 2D sum by a 2D integral) our two parameters $\nu$ and $\sigma$ are determined purely in terms of the KSE parameter $\lambda$.

\subsection{Numerical Methods}
We performed our simulations in MATLAB (version R2019a).  The domain was a periodic square, $\Omega = [-\pi,\pi)^2$, using standard pseudo-spectral methods respecting the 2/3's dealiasing rule for (see, e.g., \cite{Canuto_Hussaini_Quarteroni_Zang_2006,Orszag_1971_dealiasing,Peyret_2013_spectral_book,Shen_Tang_Wang_2011} and the references therein for details of psuedospectral methods).  (The 2/3's dealising cut-off can be seen in Figure \ref{compare_spec} as a vertical line.)  
We use an implicit/explicit Runge-Kutta-4-type algorithm, where the linear terms are handled implicitly via an exponential time-differencing algorithm (ETD, also called the exponential integrator method) using complex contour integration to handle removable singularities of the form $(e^z-1)/z$, $(e^z-z-1)/z^2$, and so on (see, e.g., \cite{Kassam_Trefethen_2005,Kennedy_Carpenter_2003_IMEX}).  

For any simulation one must of course decide on specific values of parameters that hopefully give a reasonable picture of the more general dynamics (unless the parameters are determined, e.g., by physical considerations).  This choice is typically limited by computational constraints, which limit spatial and temporal resolution.  As the KSE have a fourth-order dissipation term, it is somewhat forgiving in terms of spatial resolution, as small scales a dissipated quickly.  This is less true for the r-KSE (since the $u_1$ equation has only second-order dissipation), but one can still observe simulations which appear highly nonlinear chaotic at resolution $512^2$, reasonable enough to be run on a good laptop; a choice that was made in the hope that this study may be reproduced with relative ease by other researchers. We considered our simulations to be ``well-resolved'' if the energy spectrum of the solution had decayed to machine precision ($\approx2.2204\times10^{-16}$ in MATLAB) before the dealiasing cut-off, verified \textit{a posteriori}.  A trial-and-error search through parameter space, making sure to respect this criterion, yielded $\lambda  = 5.01$ to be a well-resolved value (for the time interval simulated), meaning there were $20=|\set{(m,n)\in\nZ^2\setminus(0,0)|m^2+n^2<5.01}|$ unstable modes in the KSE part.  Using this with \eqref{nu_sigma_values}, we find $\nu=1.67$, $\sigma=8.37$ for the r-KSE.  We also observed the case $\nu=1.67$ with $\sigma=0$ to observe the effect of the $\sigma$ term.  The time step was chosen to respect the advective CFL condition at each time step (we used a conservative value of $\Delta t\approx3.0557\times10^{-5}$).  In all simulations of r-KSE, the spatial resolution was $512^2$ grid points (uniform rectangular mesh).  For KSE simulations, the dissipation from the biLaplacian was large enough that we only needed $128^2$ resolution.  Our initial data was chosen similarly to be the well-studied initial data in \cite{Kalogirou_Keaveny_Papageorgiou_2015}.  
%
 Namely, we set 
\begin{align}\label{initial_data_sim}
 \varphi(x,y) := C(\sin(x+y) + \sin(x) + \sin(y)),\text{ and }\bu^{in}=\nabla\varphi.
\end{align}
where $C$ is chosen so that $\|\bu^{in}\|_{L^2}=1$.  
\begin{remark}\label{grad_remark}
 Several issues arise with verification of numerical schemes for 4th-order nonlinear equations in higher dimensions.  For example, the standard method of manufactured solutions (i.e., choosing a function to be an exact solution, and using it to determine an initial condition, and an appropriate forcing function on the right-hand side) can have lead to large computed errors if one uses the $L^2$ norm to compute the error.  To see this, consider a spatial resolution of $512^2$ on the domain $[-\pi,\pi)^2$ as in our simulations, meaning that the highest resolved frequency (the Nyquist frequency) is $k_{\text{Ny}}=512/2$.  Assuming a machine-zero error of $\varepsilon=2.2204\times10^{-16}$ occurs at this frequency, the resulting computation for the bi-Laplacian $\triangle^2$ for just this node would involve an error of size $\varepsilon k_{\text{Ny}}^4\approx 9.54\times10^{-7}$ (compare with the Laplacian case: $\varepsilon k_{\text{Ny}}^2\approx 1.46\times10^{-11}$).  Given that there are $512^2 = 262,144$ spatial nodes, errors can accumulate quite rapidly if one sums over the domain; hence, even if the computatiocn is done to high precision (e.g., using ETD methods or integrating factors, so that one is multiplying by factors involving small factors such as $e^{(-|\mathbf{k}|^4+\lambda |\mathbf{k}|^2)\Delta t}$), the computation of the error itself may show low precision.  Hence, seems to be better to consider, e.g., the $L^\infty$ norm instead of the $L^2$ norm for purposes of verification.  Another implication is that, if one can run at lower spatial resolution (as determined by the fall-off of the energy spectrum), it may be better to do so to avoid polluting the solution with noise.  Hence, the KSE solution we show below is run at resolution $128^2$, since the energy spectrum decays to machine precision long before the 2/3's dealising cutoff at $|\mathbf{k}| = 128/3 \approx 42.67$.
  
 Aside from the problem of computation of the error, when simulating a chaotic dynamical system such as the KSE, it is important to have several checks to make sure simulations results do not depend too heavily on the numerical scheme.  The results reported here were also checked with integrating factor methods, and similar results were obtained.  We also check that resolved simulations at lower resolution qualitatively agreed with those at higher resolution.  However, with the KSE system, we were able to perform an additional check: namely, we simulated equation \eqref{KSE_scalar} along \eqref{KSE}, resulting in solutions $\varphi$ and $\bu$ respectively, and then checked $\|\bu - \nabla\varphi\|_{L^\infty}$.  Analytically, if solutions are smooth, one should have $\bu = \nabla\varphi$, but computationally, one expects disagreement between these quantities to arise due to small errors accumulating over time, combined with the chaotic nature of the equations. 
The results of our simulations can be seen in Figure \ref{u_minus_grad_phi}.  \textit{It is for this reason that our simulations shown below are shown for relatively small times, e.g., $t\leq1$} (even though our simulations were stable for significantly larger times).   
\end{remark}
 \begin{wrapfigure}[14]{rt}{60mm}
\centering
\includegraphics[scale=0.35,trim=0mm 0mm 0mm 7mm, clip]{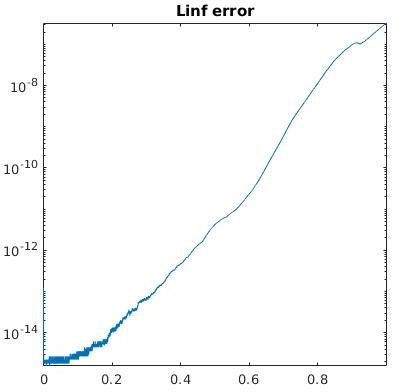}
\caption{\label{u_minus_grad_phi}\scriptsize Error 
$\|\bu(t) - \nabla\varphi(t)\|_{L^\infty}$
vs. time (Log-linear plot.)}
 \end{wrapfigure}

\subsection{Simulations}
 It is important to keep in mind that the r-KSE system \eqref{KSEr} is not meant to be a model for the KSE system \eqref{KSE} in the sense of approximating the dynamical evolution of solutions, and therefore no particular agreement between solutions is expected.  Moreover, both systems appear to behave chaotically, in the sense that small perturbations of the initial conditions or parameters can strongly affect the evolution of solutions, and therefore the major change made by moving from the the KSE system to the r-KSE system studied here is unlikely to produce similar trajectories, which is what we observe in Figure \ref{compare_solu}. However, we claim that the dynamics of the r-KSE are phenomenologically similar to the KSE, at least in certain aspects, which we investigate below.  We note that while we saw many varied types of behavior in our simulations, the simulations presented were not chosen too carefully, and we believe they represent fairly typical behavior for these systems.

  \begin{figure}[!ht]
  \phantom{aaa} 
\hspace{2mm} KSE  \hspace{28mm}r-KSE, $\nu=1.67$, $\sigma=8.37$ \hspace{0.3cm} r-KSE, $\nu=1.67$, $\sigma = 0$\\
\includegraphics[width=0.32\textwidth,trim=28mm 8mm 27mm 4mm, clip]{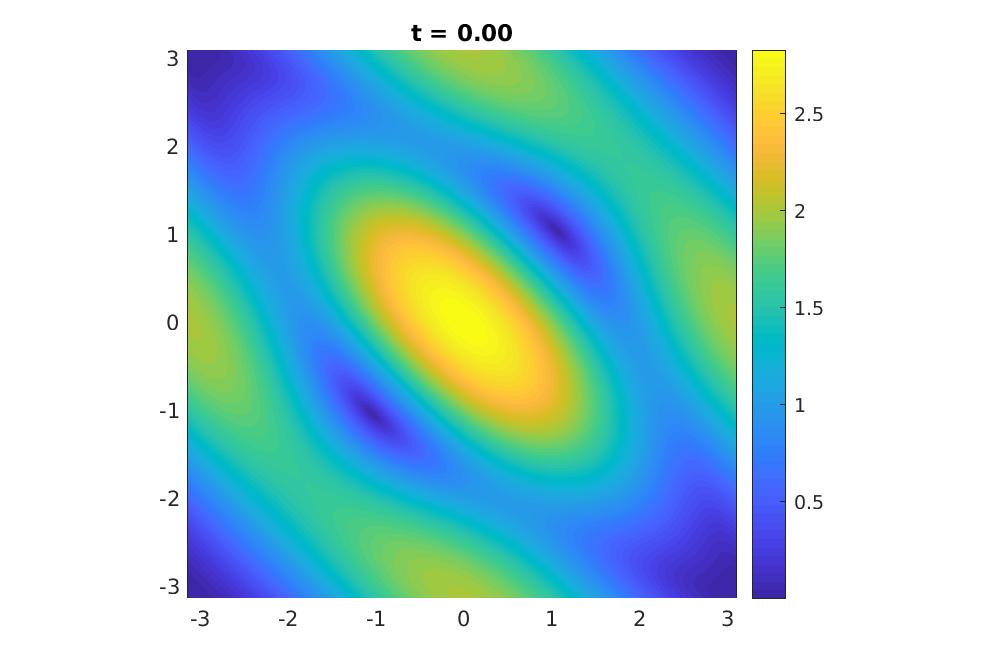}
\includegraphics[width=0.32\textwidth,trim=28mm 8mm 27mm 4mm, clip]{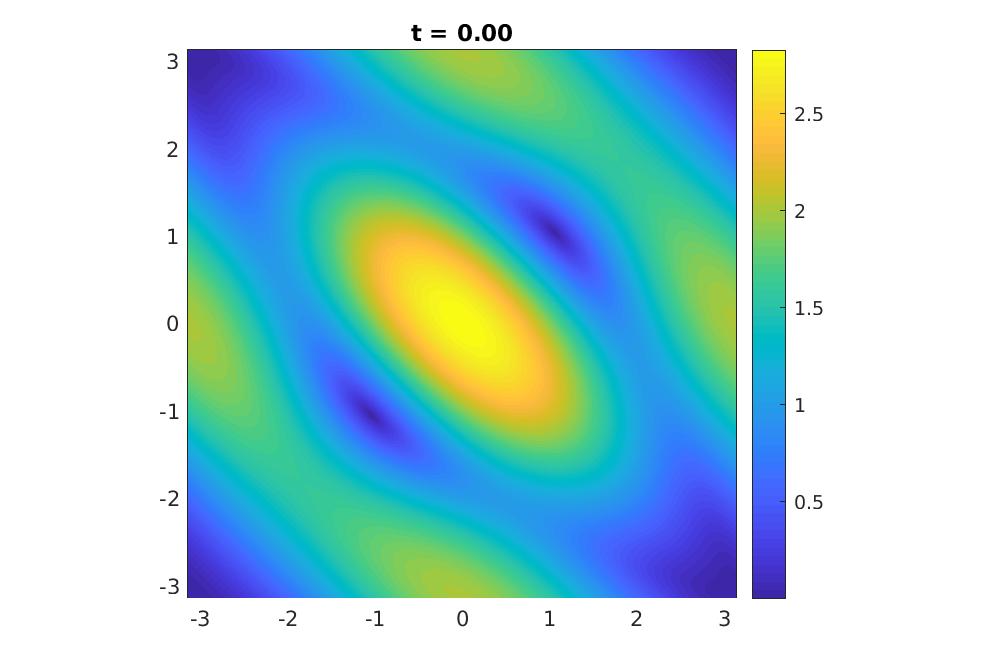}
\includegraphics[width=0.32\textwidth,trim=28mm 8mm 27mm 4mm, clip]{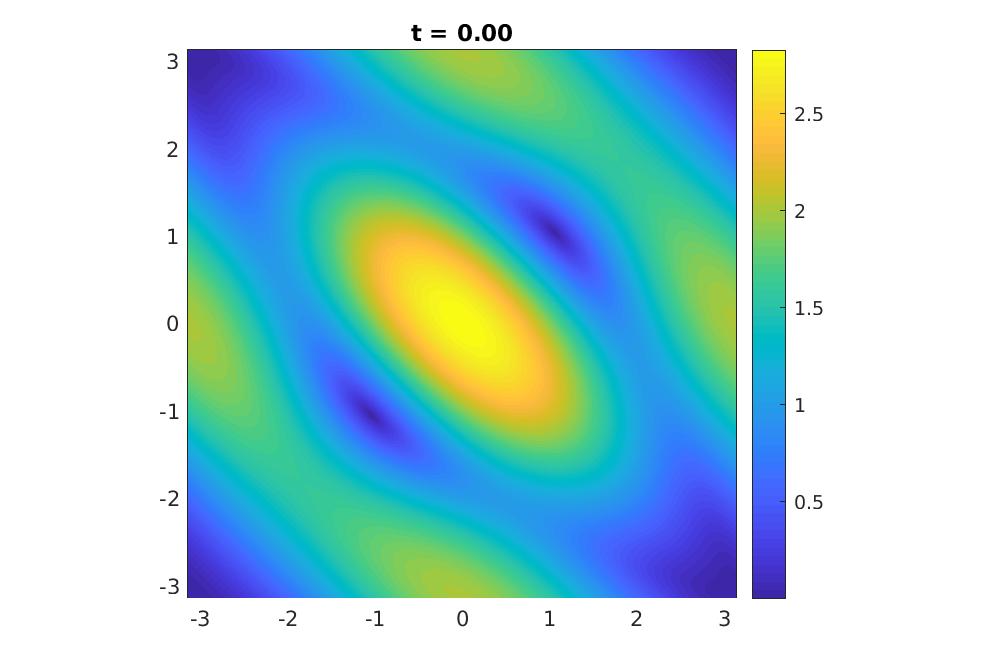}
\hspace{1cm}
\includegraphics[width=0.32\textwidth,trim=28mm 8mm 27mm 4mm, clip]{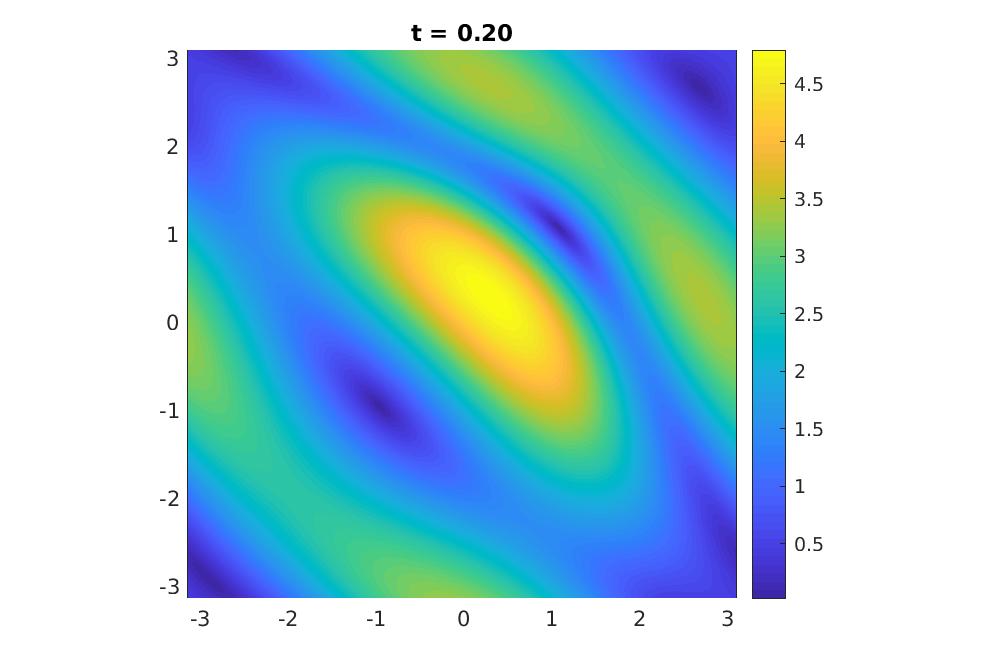}
\includegraphics[width=0.32\textwidth,trim=28mm 8mm 27mm 4mm, clip]{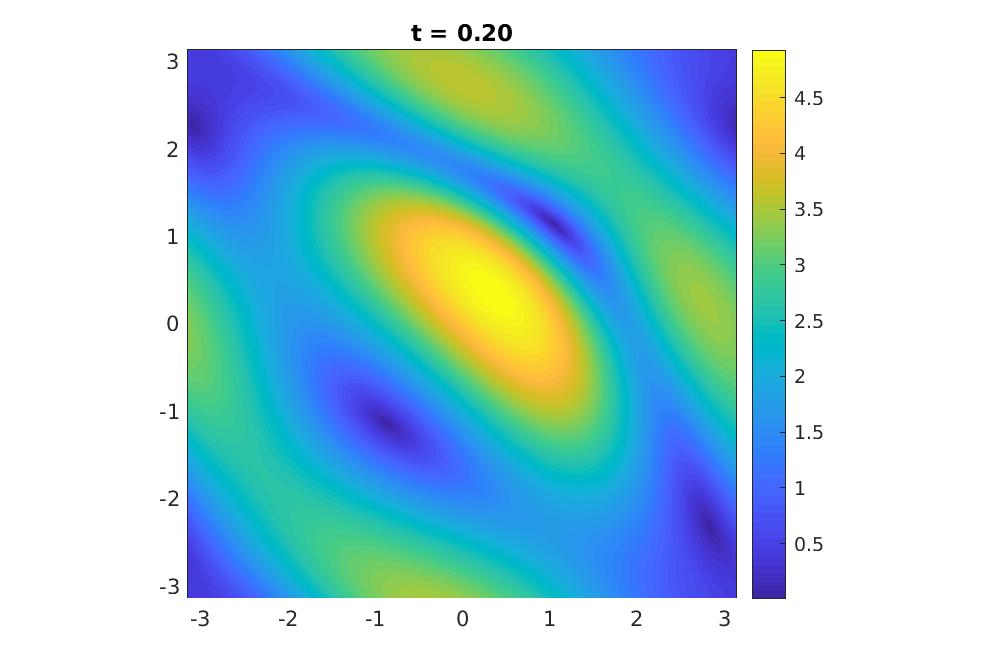}
\includegraphics[width=0.32\textwidth,trim=28mm 8mm 27mm 4mm, clip]{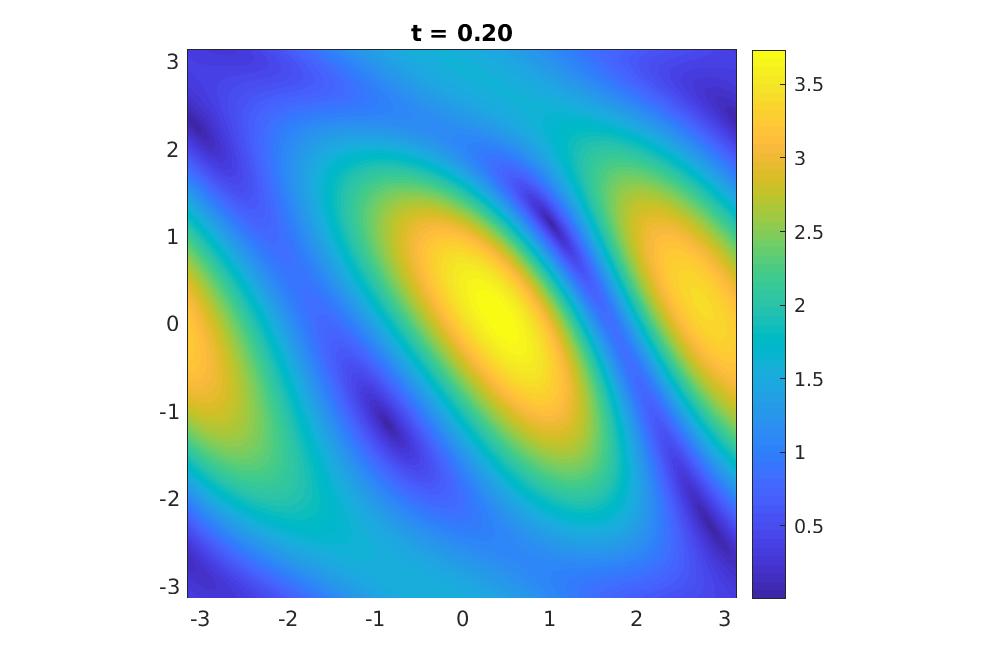}
\hspace{1cm}
\includegraphics[width=0.32\textwidth,trim=28mm 8mm 27mm 4mm, clip]{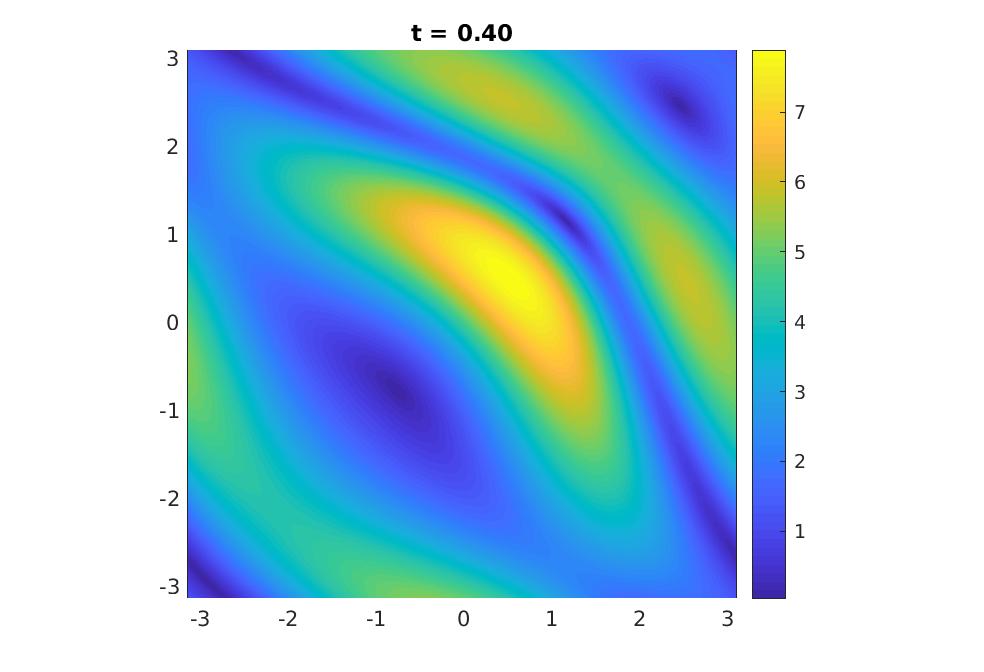}
\includegraphics[width=0.32\textwidth,trim=28mm 8mm 27mm 4mm, clip]{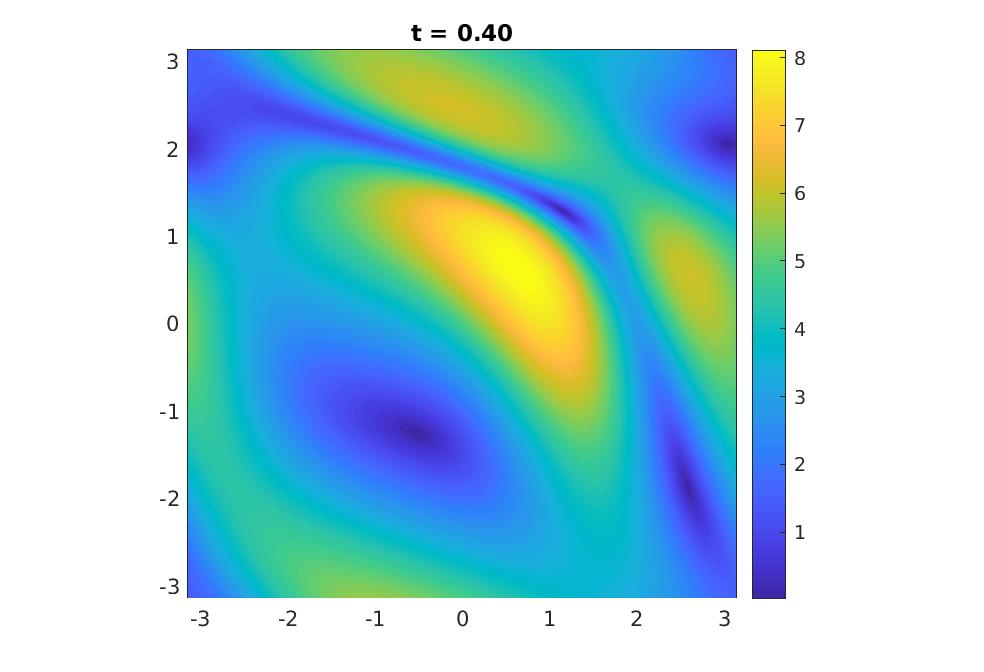}
\includegraphics[width=0.32\textwidth,trim=28mm 8mm 27mm 4mm, clip]{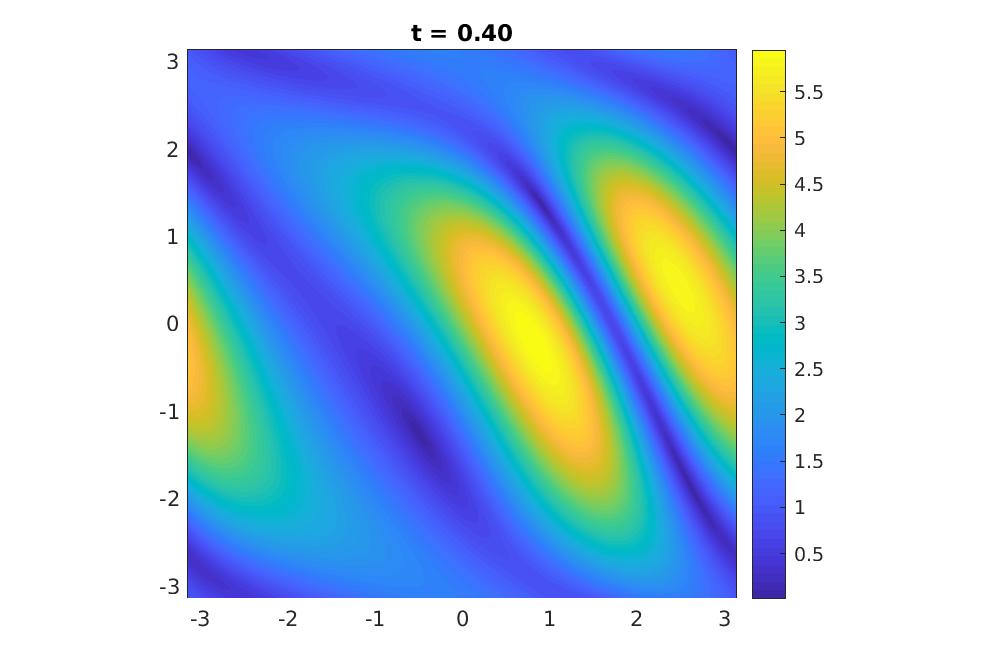}
\includegraphics[width=0.32\textwidth,trim=28mm 8mm 27mm 4mm, clip]{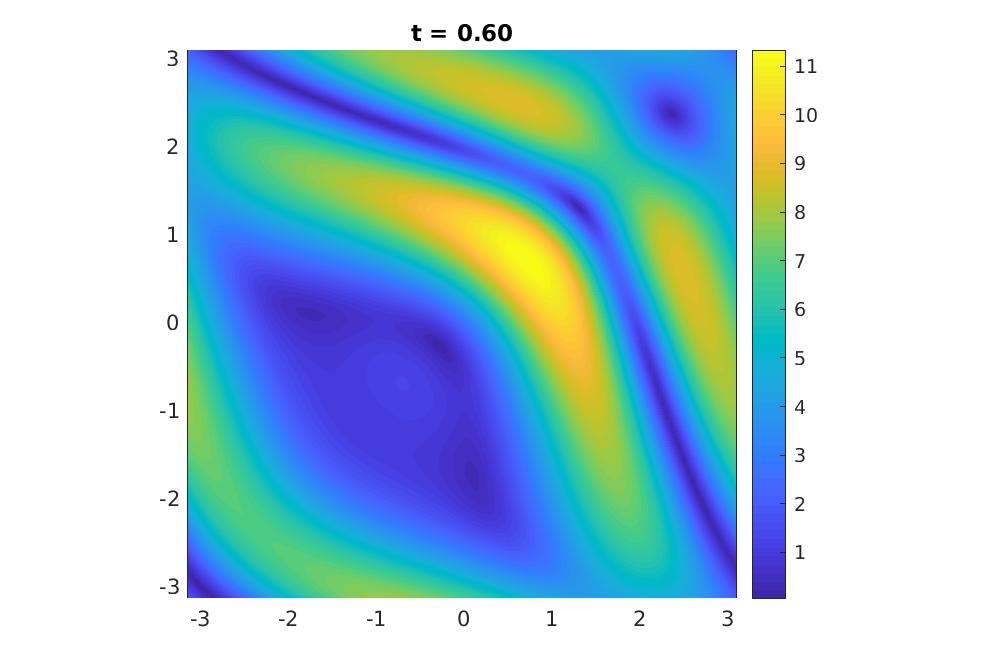}
\includegraphics[width=0.32\textwidth,trim=28mm 8mm 27mm 4mm, clip]{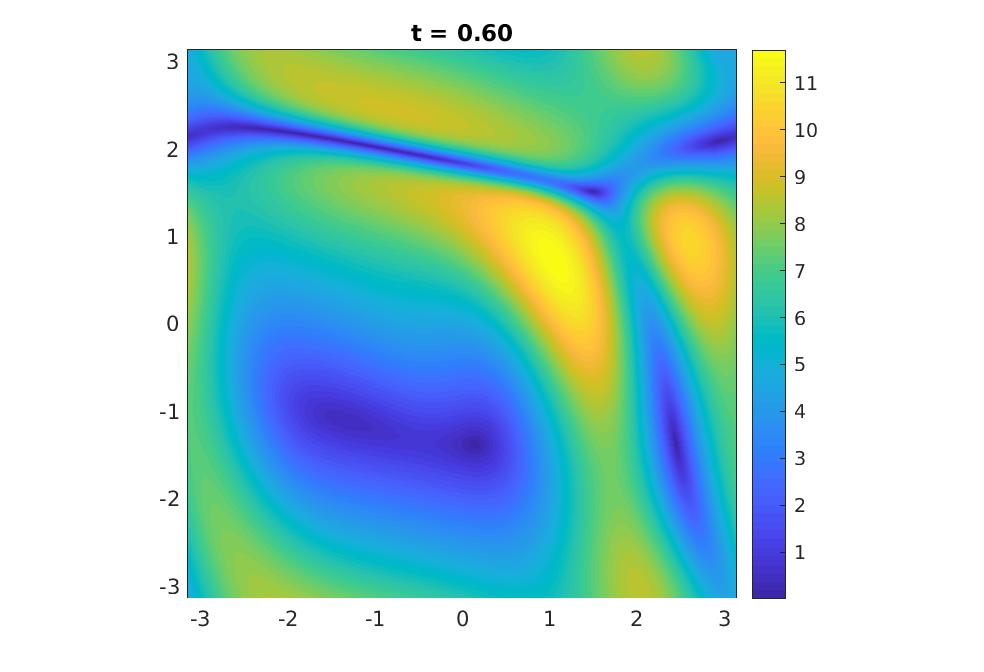}
\includegraphics[width=0.32\textwidth,trim=28mm 8mm 27mm 4mm, clip]{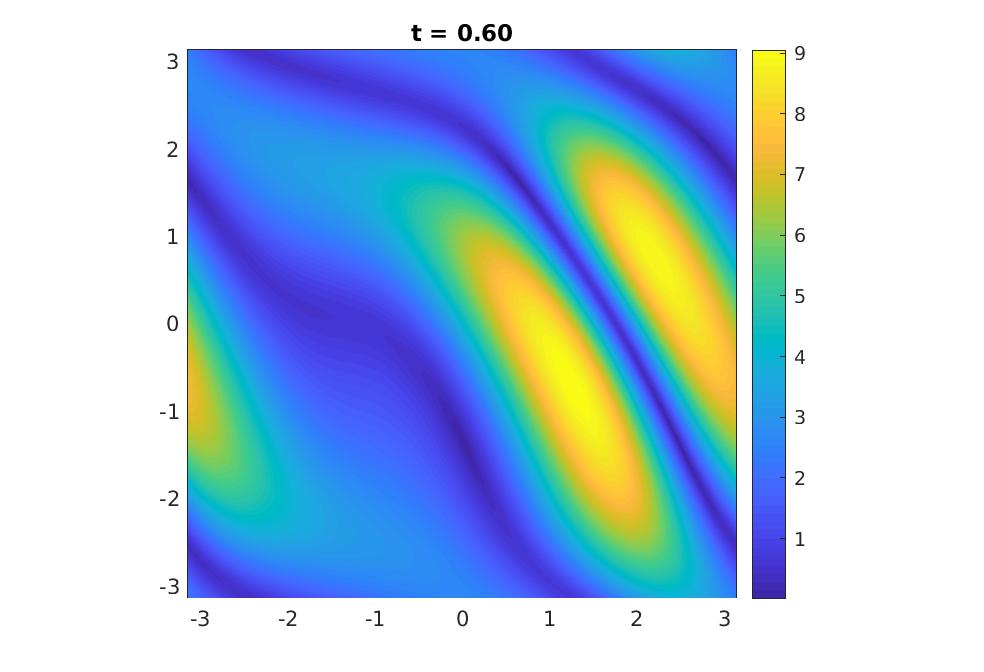}
\hspace{1cm}
\includegraphics[width=0.32\textwidth,trim=28mm 8mm 27mm 4mm, clip]{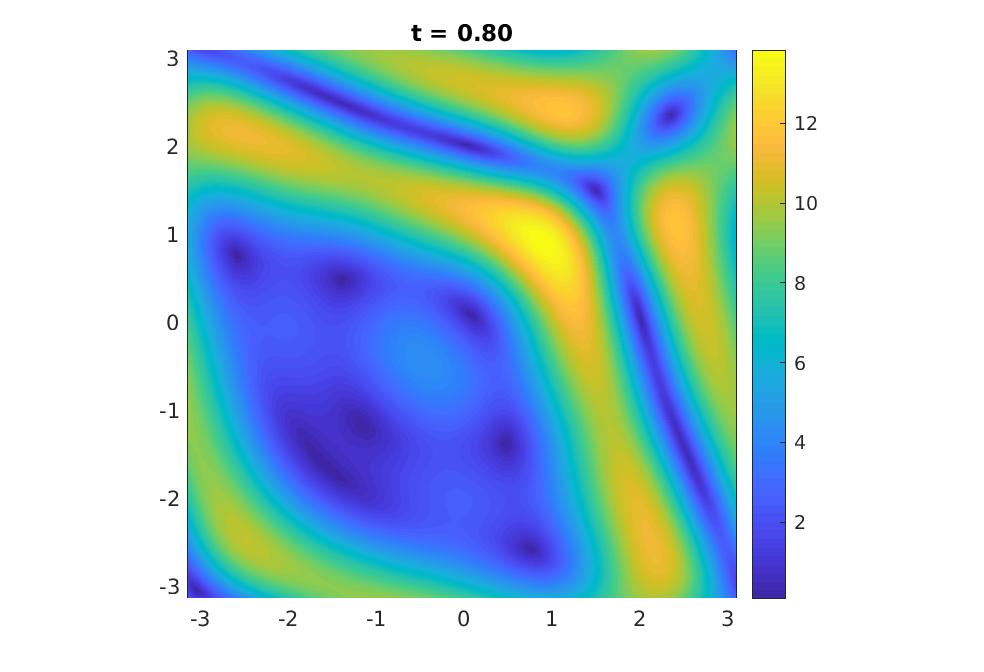}
\includegraphics[width=0.32\textwidth,trim=28mm 8mm 27mm 4mm, clip]{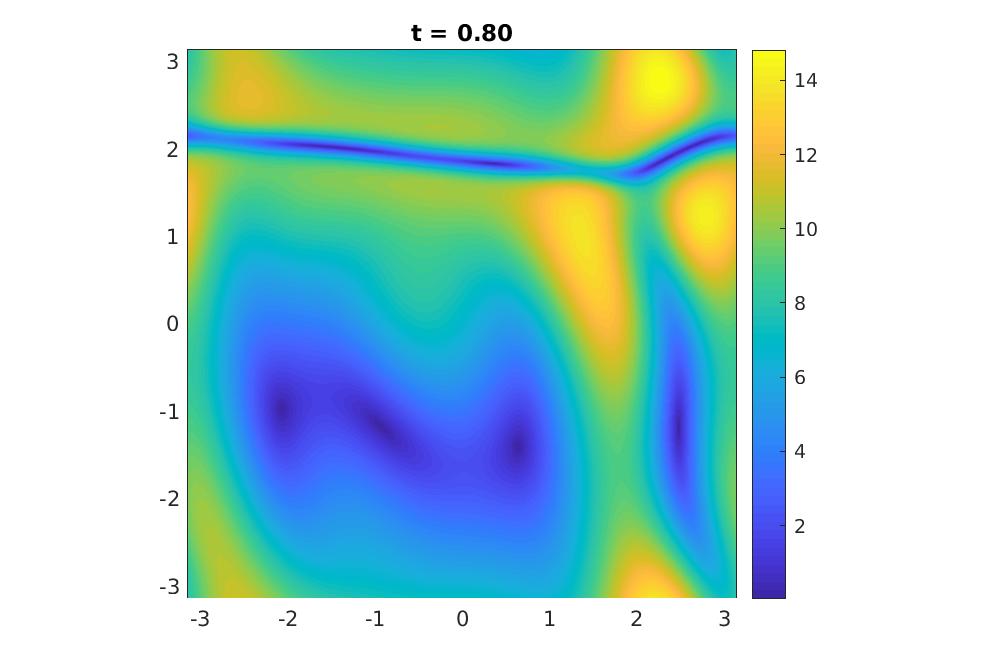}
\includegraphics[width=0.32\textwidth,trim=28mm 8mm 27mm 4mm, clip]{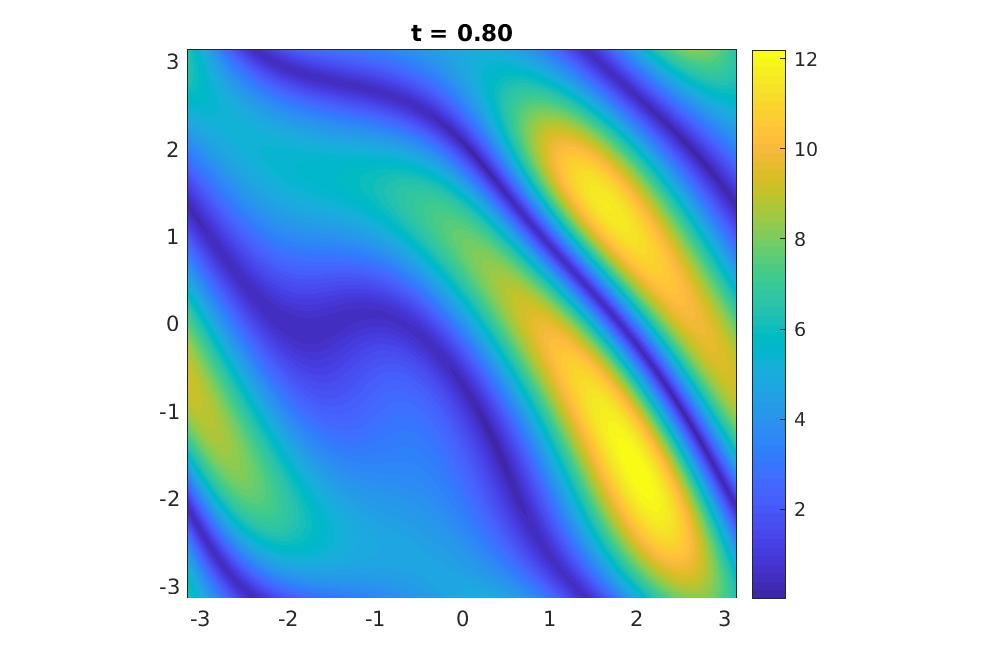}\vspace{-2mm}
	\caption{\label{compare_solu}  Solution magnitudes ($\sqrt{u_1^2+u_2^2}$) for KSE (left); r-KSE with $\nu=1.67$, $\sigma=8.37$ (middle); and r-KSE with $\nu=1.67$, $\sigma=0$ (right).  For all simulations, $\lambda = 5.01$.  Times (top to bottom): 0.0, 0.2, 0.4, 0.6, 0.8.}  
\end{figure}

\begin{figure}[!ht]
\hspace{8mm} KSE  \hspace{22mm}r-KSE, $\nu=1.67$, $\sigma=8.37$ \hspace{0.3cm} r-KSE, $\nu=1.67$, $\sigma = 0$\\
\includegraphics[width=0.32\textwidth,trim=0mm 6mm 48mm 4mm, clip]{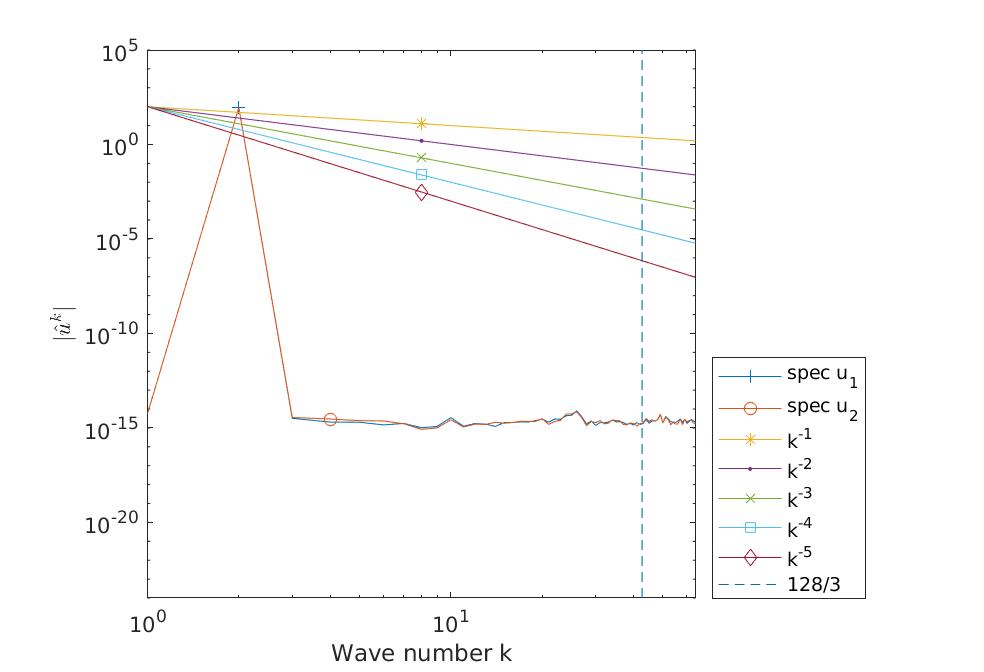}
\includegraphics[width=0.32\textwidth,trim=0mm 6mm 48mm 4mm, clip]{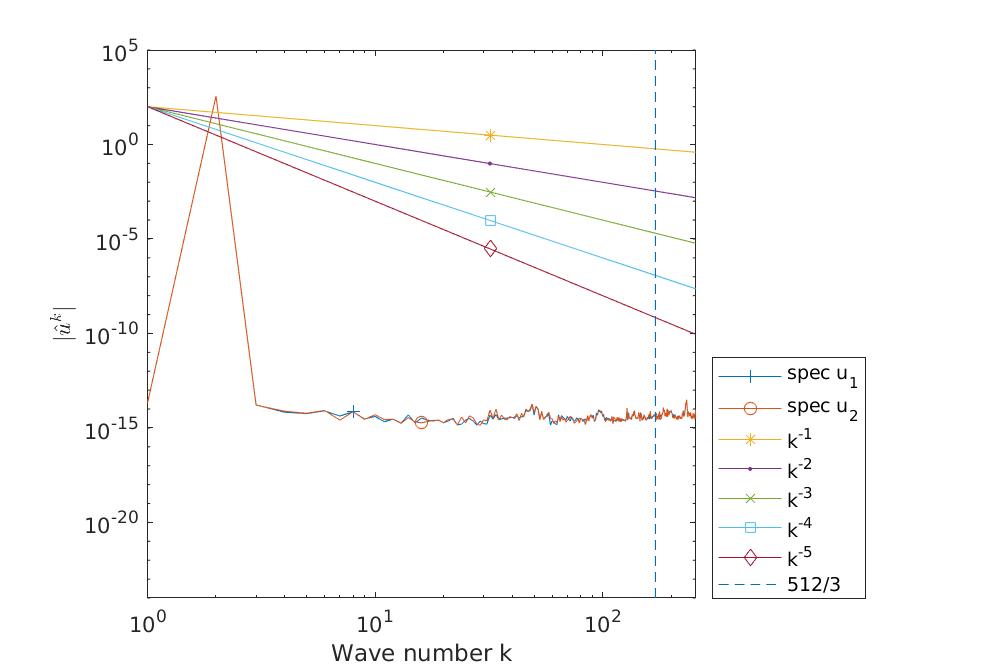}
\includegraphics[width=0.32\textwidth,trim=0mm 6mm 19mm 4mm, clip]{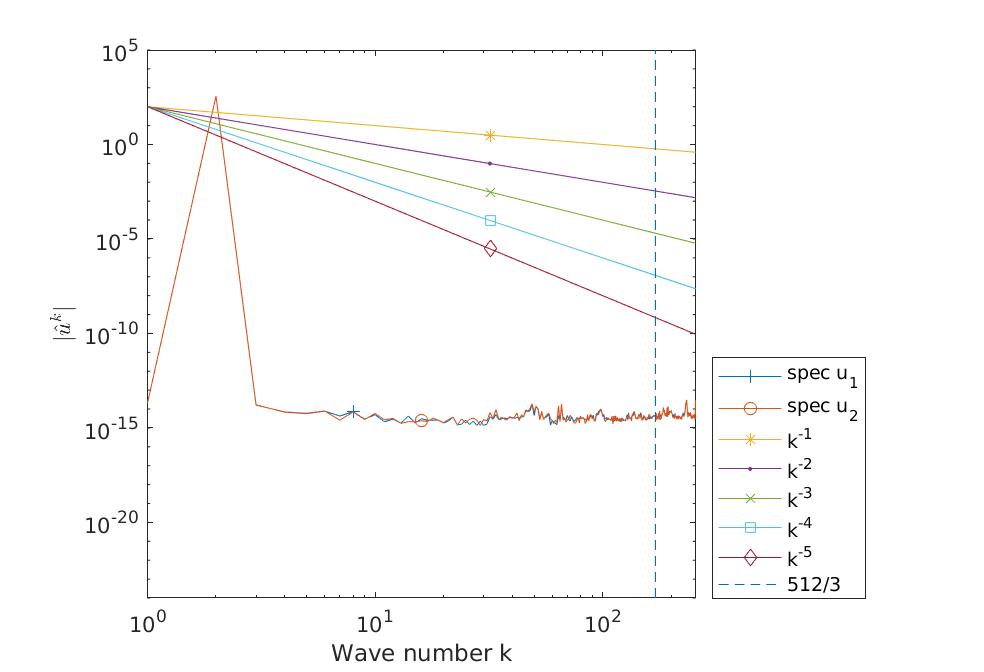}
\hspace{1cm}
\includegraphics[width=0.32\textwidth,trim=0mm 6mm 48mm 4mm, clip]{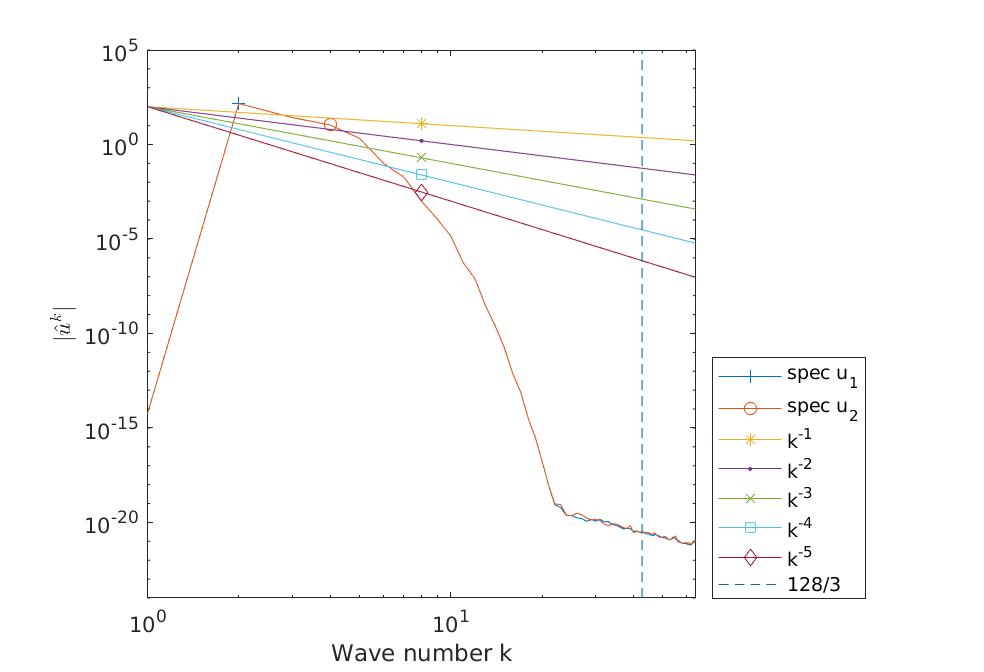}
\includegraphics[width=0.32\textwidth,trim=0mm 6mm 48mm 4mm, clip]{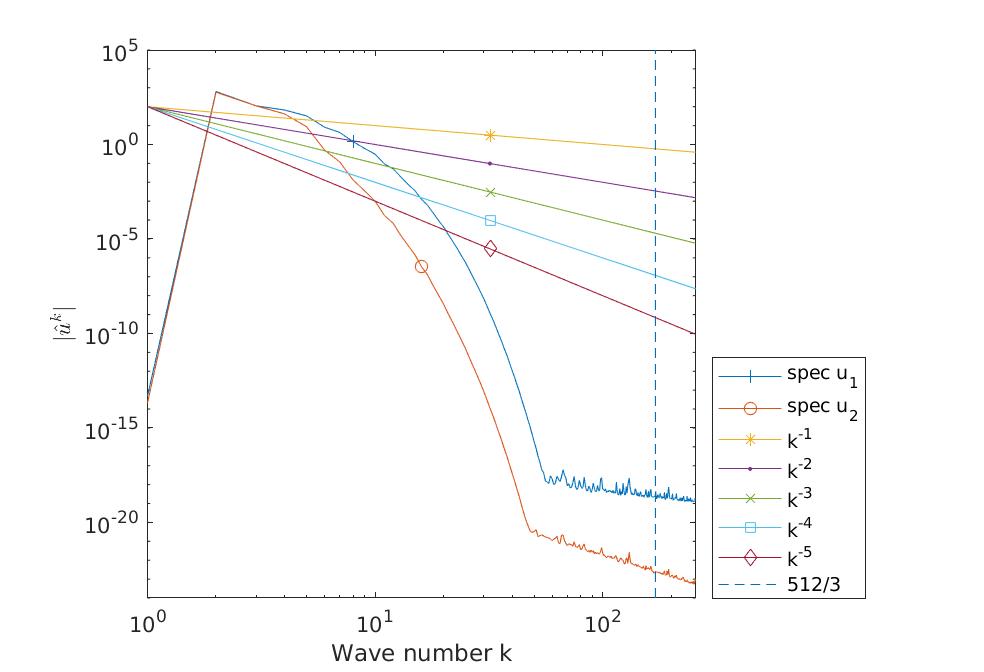}
\includegraphics[width=0.32\textwidth,trim=0mm 6mm 19mm 4mm, clip]{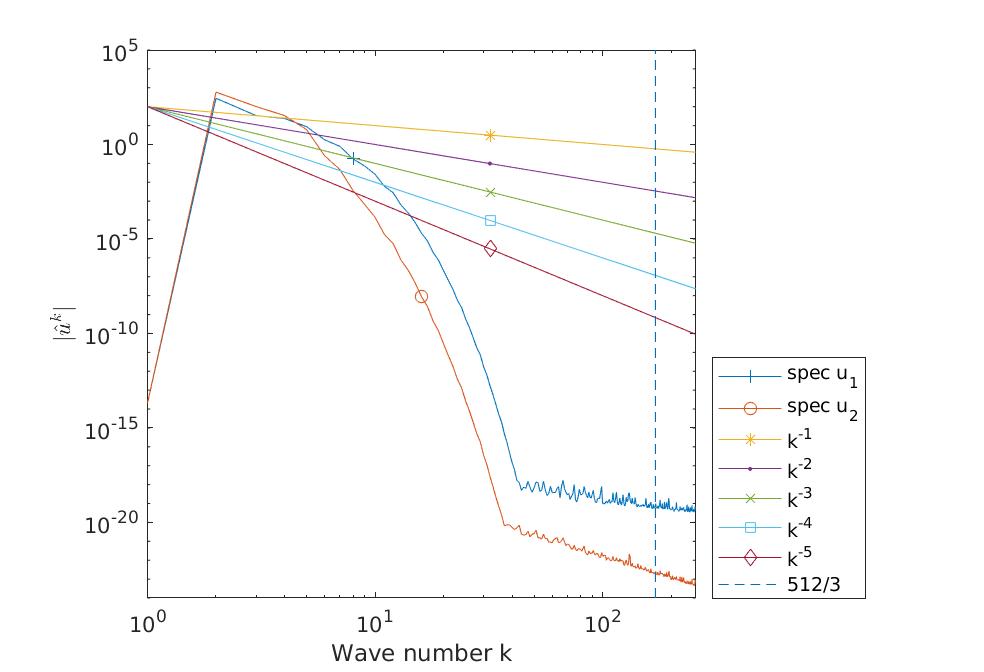}
\hspace{1cm}
\includegraphics[width=0.32\textwidth,trim=0mm 6mm 48mm 4mm, clip]{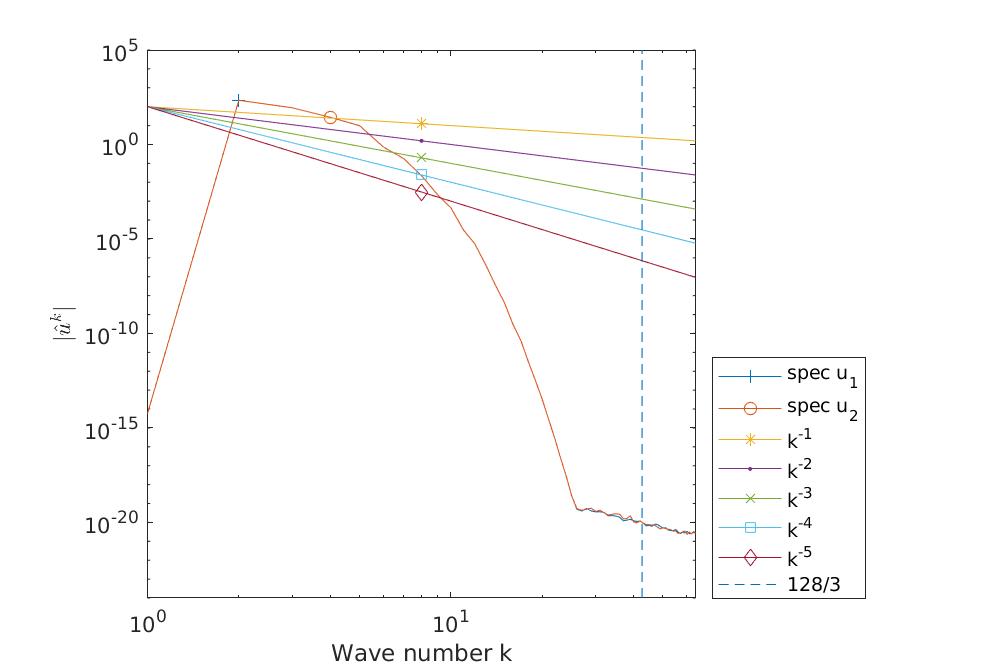}
\includegraphics[width=0.32\textwidth,trim=0mm 6mm 48mm 4mm, clip]{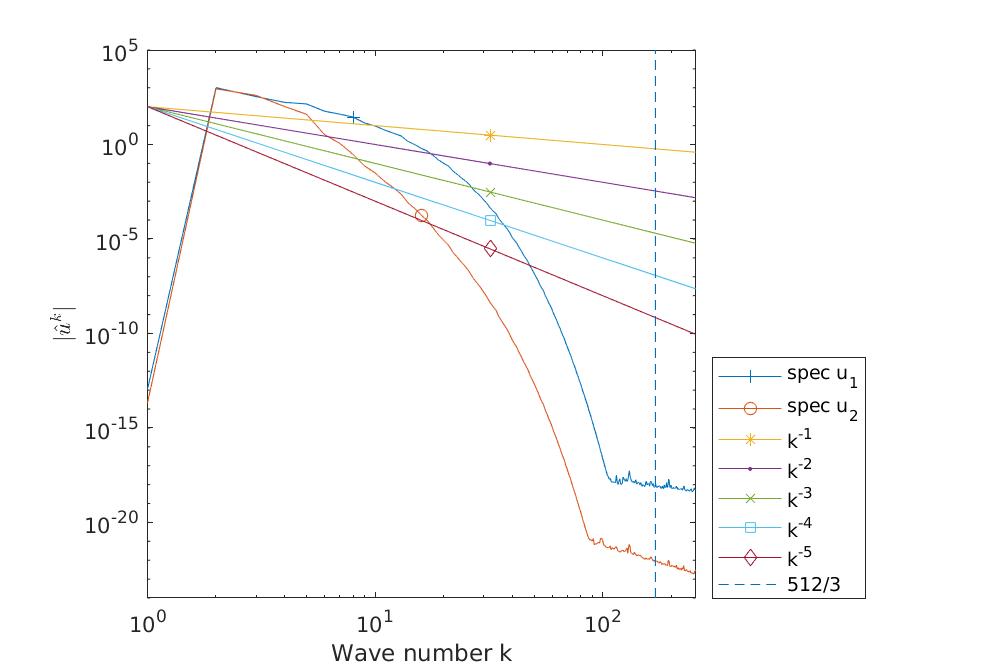}
\includegraphics[width=0.32\textwidth,trim=0mm 6mm 19mm 4mm, clip]{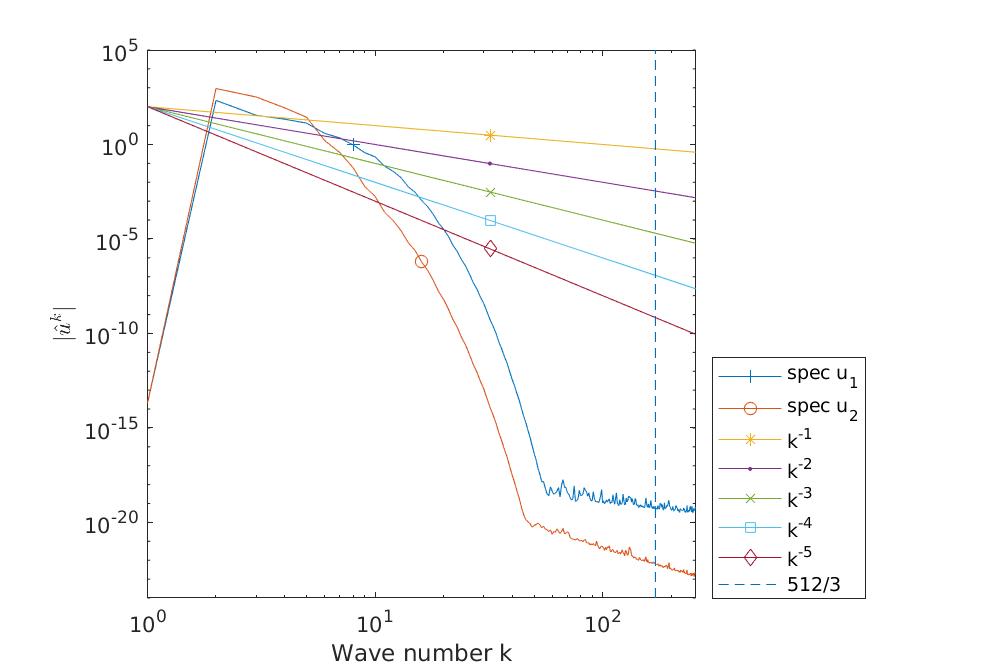}
\includegraphics[width=0.32\textwidth,trim=0mm 6mm 48mm 4mm, clip]{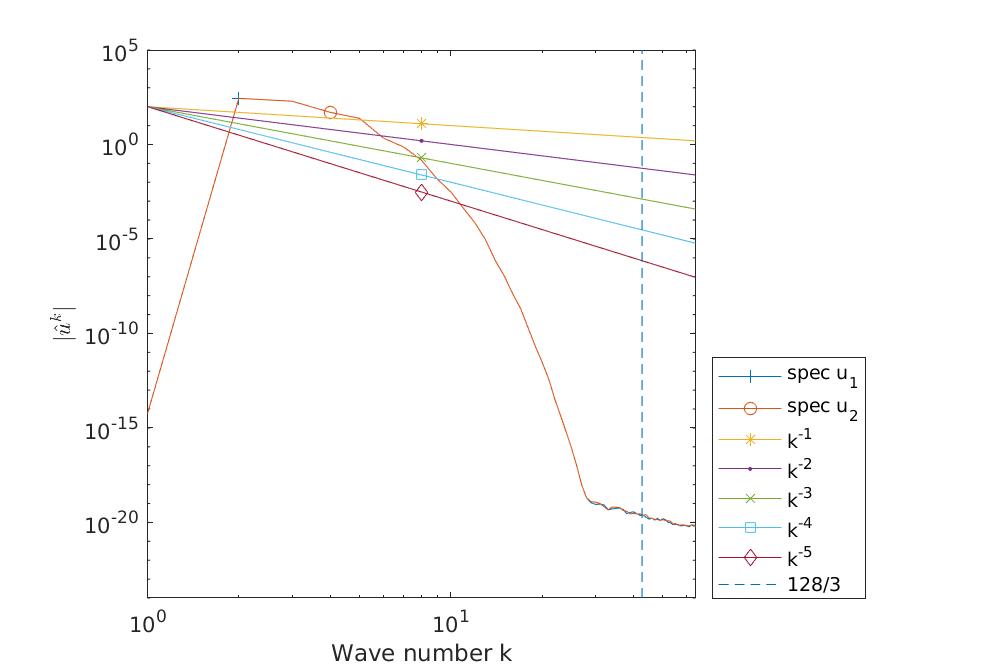}
\includegraphics[width=0.32\textwidth,trim=0mm 6mm 48mm 4mm, clip]{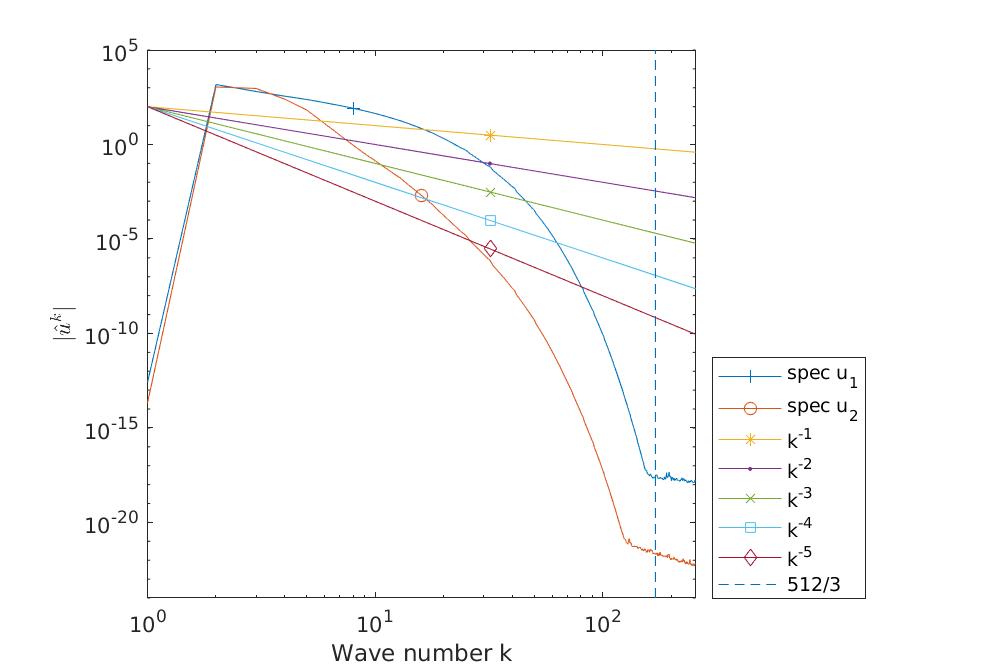}
\includegraphics[width=0.32\textwidth,trim=0mm 6mm 19mm 4mm, clip]{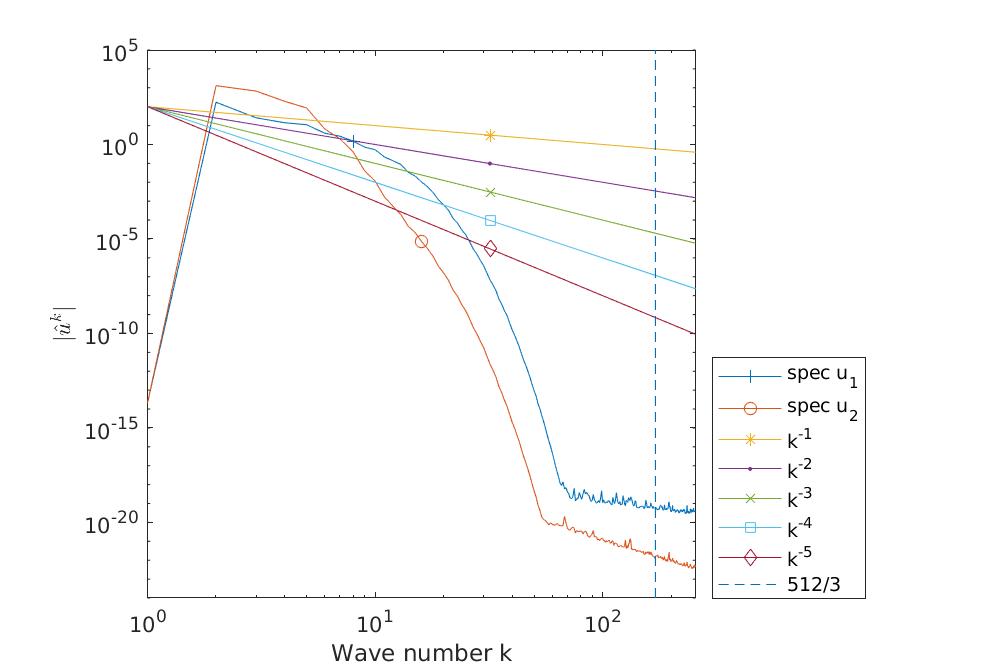}
\hspace{1cm}
\includegraphics[width=0.32\textwidth,trim=0mm 0mm 48mm 4mm, clip]{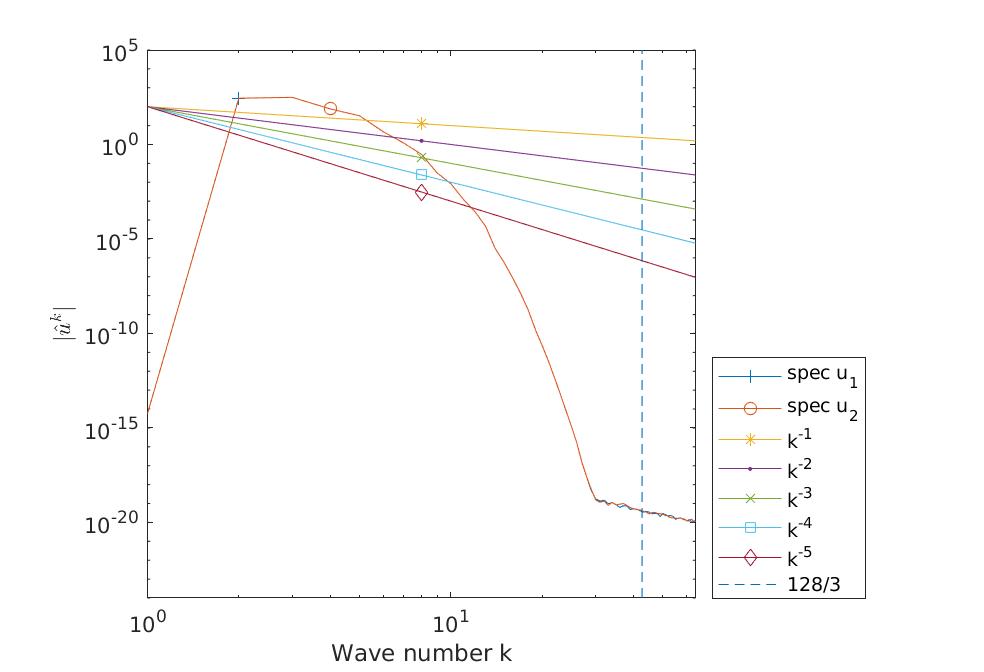}
\includegraphics[width=0.32\textwidth,trim=0mm 0mm 48mm 4mm, clip]{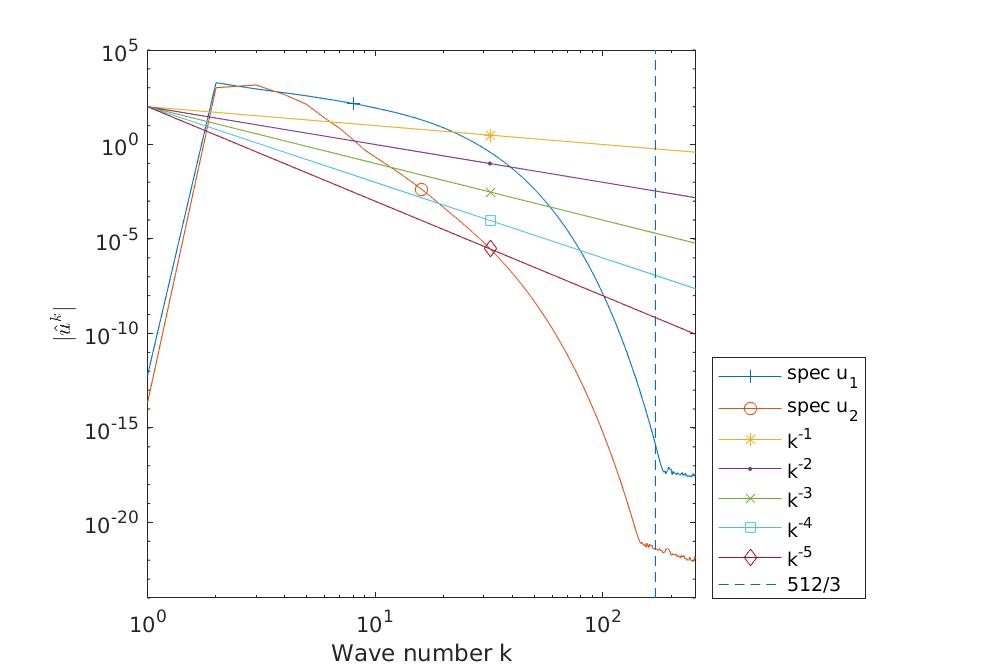}
\includegraphics[width=0.32\textwidth,trim=0mm 0mm 19mm 4mm, clip]{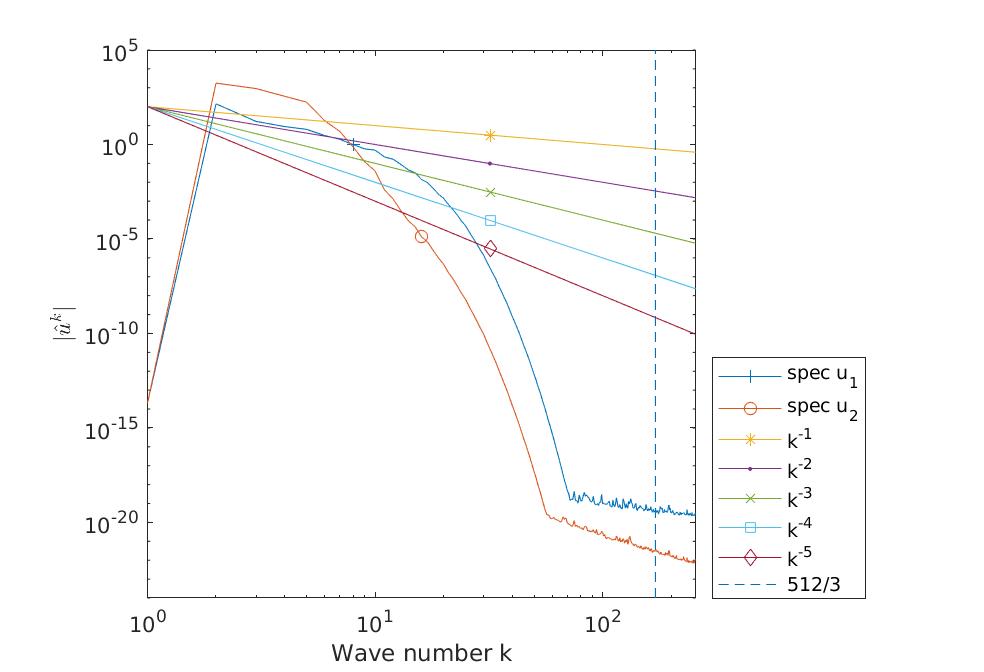}
	\caption{\label{compare_spec} Energy spectra for KSE (left); r-KSE with $\nu=1.67$, $\sigma=8.37$ (middle); and r-KSE with $\nu=1.67$, $\sigma=0$ (right).  For all simulations, $\lambda = 5.01$.  Times (top to bottom): 0.0, 0.2, 0.4, 0.6, 0.8.}  
\end{figure}

  As expected, in  Figure \ref{compare_solu} there are clear differences, both quantitative and qualitative, between the solutions in both systems.  Thus, we do not claim that the solutions of the r-KSE are reasonable approximations to the KSE.    However, a closer look does reveal some qualitative agreements.  We note that similar length-scales develop at approximately the same time, and also solution amplitudes grow at roughly the same rate.  Both systems develop new cell-like structures, although they appear to be more complex in the KSE case.  We also observed in large-time simulations (not shown here) that solutions to the KSE and r-KSE often move toward a quasi-one-dimensional state, a phenomenon was investigated in the context of the 2D KSE in \cite{Kalogirou_Keaveny_Papageorgiou_2015}.   The r-KSE solutions tend to approach this state more rapidly, perhaps due increased smoothing in one direction and anisotropic instability (although the orientation of the 1D state, vertical or horizontal, seems to be highly sensitive to small perturbations in the initial data and parameters).  

  
Qualitative similarities are also to be found in the energy spectrum.  We include these in part in order to show that the simulations are well-resolved over the time interval in question.  However, we note that in Figure \ref{compare_spec}, we see that the spectrum for $u_1$ and $u_2$ are right on top of each other for the KSE, while the spectrum of $u_1$ takes on a somewhat different character from the spectrum of $u_2$ for the r-KSE.  This phenomenon was observed by the authors in many different simulations using a wide range of different parameters and initial data, although we present only one particular representative simulation here.  We note that these are \textit{not} time averaged spectra, but spectra at each time.  

\begin{remark}
Many works on the KSE have concentrated on the so-called ``equipartion of energy'' in solutions to the KSE (see, e.g., \cite{Kalogirou_Keaveny_Papageorgiou_2015,Tomlin_Kalogirou_Papageorgiou_2018,Hyman_Nicolaenko_Zaleski_1986,Sneppen_Krug_Jensen_Jayaprakash_Bohr_1992}).  By \textit{equipartion of energy}, it is usually meant that there is a range in the time-averaged spectrum where the energy is statistically equally distributed between spherically-averaged Fourier modes.  However, we do not examine such considerations here, in part because our simulation times are too short (for reasons explained in Remark \ref{grad_remark}) to allow for a reasonable time-averaging, but also because, with $\lambda=5.01$, our number of non-zero, spherically-averaged Fourier modes is only 5, which seems to be too small of a number to draw conclusions from.
\end{remark}

\FloatBarrier

\section{Conclusion}
In this study, we have shown that the 2D r-KSE is globally well-posed, that it enjoys many of the same mathematical properties as the 2D KSE (discussed in the introduction), and that computationally,  its dynamics have a qualitative resemblance to the dynamics of the KSE (e.g., the time evolution of various norms, and the spectrum of the ``unreduced'' component).  
Therefore, we believe that the 2D r-KSE has the potential to serve as an instructive phenomenological model for the 2D KSE, playing a similar role to the 3D Burgers equation for the 3D NSE.  
Indeed, this analogy is stronger than one might initially suppose: in reducing the 3D Navier-Stokes equations to the 3D Burgers equations, one removes a term (namely the pressure gradient) to allow for a maximum principle.  
This is similar to the strategy behind reducing the 2D KSE to the 2D r-KSE, although we actually did not allow a maximum principle (except when $\sigma=0$) but only an exponential growth bound . 

Much like the 3D NSE, the 2D KSE is not known to be globally well-posed for arbitrary smooth initial data.  However, we note that there exists a wide variety of globally well-posed models that are phenomenologically similar to the 3D NSE (e.g., the 3D Navier-Stokes-$\alpha$ model, the 3D viscous Burgers equation, 3D NSE with hyperviscosity,  etc.) that can lead to useful insights about the 3D NSE, serve as instructive counter-examples, and guide new research directions.  In contrast, we know of no such globally well-posed analogues for the 2D KSE, other than the 1D KSE or models where the nonlinearity is essentially one-dimensional (which clearly have strong differences from the 2D KSE), and the r-KSE model proposed here.  The aim of the present work has been to provide a system which can act as such an analogue.

\section*{Acknowledgments}
\noindent The authors would like to thank the editor and anonymous referees for helpful comments and suggestions. Author A.L. would also like to thank Edriss S. Titi for initially inspiring an interest in the 2D Kuramoto-Sivashinsky equations.  Author A.L. was partially supported by NSF grants DMS-1716801 and CMMI-1953346. 
 
\begin{scriptsize}

\end{scriptsize}

\end{document}